\documentclass[11pt,a4paper]{article}
\usepackage[english]{babel}
\usepackage{amssymb}
\usepackage{amsmath}
\usepackage{amsthm}
\usepackage[a4paper,margin=2.54cm]{geometry}
\usepackage{graphicx}
\usepackage{todonotes}
\usepackage{appendix}
\usepackage{todonotes}

\theoremstyle{plain}
\newtheorem*{theorem*}{Theorem}
\newtheorem{theorem}{Theorem}[section] 
\newtheorem{lemma}[theorem]{Lemma}
\newtheorem{proposition}[theorem]{Proposition}
\newtheorem{corollary}[theorem]{Corollary}
\newtheorem{conjecture}[theorem]{Conjecture}

\theoremstyle{definition}
\newtheorem{definition}[theorem]{Definition}
\newtheorem{remark}[theorem]{Remark}

\makeatletter
\@addtoreset{equation}{section}
\makeatother
\numberwithin{equation}{section}

\DeclareMathOperator{\im}{Im}
\DeclareMathOperator{\re}{Re}
\DeclareMathOperator{\Wr}{Wr}

\usepackage{authblk}
\title{Exceptional Jacobi polynomials}
\author[1]{Niels Bonneux}
\affil[1]{Katholieke Universiteit Leuven, Department of Mathematics, 
	Celestijnenlaan~200B box 2400, 3001 Leuven, Belgium. 
	E-mail:~{\tt niels.bonneux@kuleuven.be}
}
\date{\today}
\providecommand{\keywords}[1]{\textit{\textit{Keywords: }} #1}

\usepackage{hyperref}
\AtBeginDocument{} 

\newcommand*\pFqskip{8mu}
\catcode`,\active
\newcommand*\pFq{\begingroup
	\catcode`\,\active
	\def ,{\mskip\pFqskip\relax}%
	\dopFq
}
\catcode`\,12
\def\dopFq#1#2#3#4#5{%
	{}_{#1}F_{#2}\biggl[\genfrac..{0pt}{}{#3}{#4};#5\biggr]%
	\endgroup
}

\begin{document}

\maketitle


\begin{abstract}
In this paper we present a systematic way to describe exceptional Jacobi polynomials via two partitions. We give the construction of these polynomials and restate the known aspects of these polynomials in terms of their partitions. The aim is to show that the use of partitions is an elegant way to label these polynomials. Moreover, we prove asymptotic results according to the regular and exceptional zeros of these polynomials.
\end{abstract}



\section{Introduction}\label{sec:introduction}

In 1929, Bochner classified the polynomials which are eigenpolynomials of a second order operator \cite{Bochner}. Later, Lesky and others described the orthogonality of these polynomials \cite{Lesky}. In the last decade, this theory has been extended to exceptional orthogonal polynomials. These new polynomials are obtained via a series of Darboux transformations (or Darboux-Crum transformation) \cite{Crum,Darboux} starting with a second order operator described by Bochner's classification. The new eigenpolynomials generalize their classical counterparts yet there is one striking difference: the exceptional polynomials have gaps in their degree sequence. Stated differently, there is no exceptional polynomial for every degree. Remarkably, in specific situations these exceptional polynomials still form a complete set of orthogonal polynomials. 

A significant step towards a full classification of exceptional orthogonal polynomials is discussed in \cite{GarciaFerrero_GomezUllate_Milson}. The authors proved that there are only three kinds of exceptional orthogonal polynomials: exceptional Hermite, exceptional Laguerre and exceptional Jacobi polynomials. The exceptional Hermite setting is the most studied, and a full description is given in \cite{GomezUllate_Grandati_Milson-b}. One can also approach these polynomials by taking the limit of Casorati determinants of Charlier polynomials, see \cite{Duran-Hermite}. The orthogonality for the Laguerre situation is covered in \cite{Duran-Laguerre1,Duran_Perez-Laguerre2}. In the case of exceptional Jacobi polynomials, sufficient conditions for orthogonality are given in \cite{Duran-b}. They seem to be necessary too, up to an extra technical condition. Hence the orthogonality is close to being finished.

In this paper, we enter the Jacobi setting. We give a description of how these polynomials are built and inspect the issue of the asymptotic behavior of their zeros. For the Laguerre case, we refer to \cite{Bonneux_Kuijlaars} where the same aspects are treated. The exceptional Hermite polynomials are discussed in \cite{GomezUllate_Grandati_Milson-b}, the asymptotic behavior of their zeros is explained in \cite{Kuijlaars_Milson}.

Jacobi polynomials are well-investigated polynomials \cite{Szego} and are named for the German mathematician Carl Jacobi. The first mention of exceptional Jacobi polynomials can be found around ten years ago in \cite{GomezUllate_Kamran_Milson-09,Sasaki_Tsujimoto_Zhedanov}. Nowadays, there are several papers that treat different aspects of these polynomials \cite{Grandati_Berard,Dimitrov_Lun,Duran-b,Ho_Odake_Sasaki,Liaw_Littlejohn_Stewart_Wicks,Liaw_Stewart_Osborn,Midya_Roy,Sasaki_Tsujimoto_Zhedanov,Takemura-Heun}. Moreover, some authors refer to these polynomials as multi-indexed Jacobi polynomials \cite{Ho_Sasaki_Takemura,Odake_Sasaki-a,Odake_Sasaki-b,Takemura-Maya}. Most of these papers just consider a specific case of exceptional Jacobi polynomials, for example $X_m$-Jacobi polynomials. Therefore, a first goal of this paper is to give a description of the construction of exceptional Jacobi polynomials.
We construct the exceptional Jacobi polynomial via partitions \cite{Andrews}, as we believe this is the most elegant way to approach exceptional polynomials. For exceptional Hermite polynomials, the use of 1 partition is sufficient to capture all possibilities \cite{GomezUllate_Grandati_Milson-b}. For Laguerre polynomials, 2 partitions are needed \cite{Bonneux_Kuijlaars}. Like in the Laguerre setting, the exceptional Jacobi polynomials are obtained via two partitions so that we identify the $X_m$-Jacobi polynomials as a specific choice of these partitions. 

The zeros of the Jacobi polynomials are well-investigated. When these polynomials are orthogonal on the interval $[-1,1]$, all the zeros are simple and lie in the open interval. For exceptional Jacobi polynomials, the zeros can be outside this interval or lie in the complex plane. For special cases of exceptional Jacobi polynomials, the behavior has already been studied in \cite{GomezUllate_Marcellan_Milson,Ho_Sasaki,Horvath}. As we consider a general approach, we cover most of these known results. We prove that the number of regular zeros tends to infinity as the degree tends to infinity. Next, the asymptotic behavior of these zeros is related to the Bessel function and the Arcsine distribution. For the exceptional zeros we derive that these zeros are attracted by the simple zeros of the generalized Jacobi polynomial (which is defined as a specific Wronskian). These results were conjectured in \cite{Kuijlaars_Milson}, where the authors tackled the Hermite case. Their techniques are now transferred to the Jacobi case where the Laguerre case was elaborated in \cite{Bonneux_Kuijlaars}. Finally, we end with a conjecture dealing with simple zeros. We believe that the generalized Jacobi polynomials have simple zeros when the corresponding exceptional Jacobi polynomials form a complete set of orthogonal polynomials. This conjecture is comparable to the Veselov conjecture, which deals with simple zeros of the Wronskian of an arbitrary (finite) sequence of Hermite polynomials \cite{Felder_Hemery_Veselov}, or to the conjecture in the Laguerre case \cite{Bonneux_Kuijlaars}. 

We composed the paper as follows. In section \ref{sec:XJP} we define the generalized and exceptional Jacobi polynomials properly. Section \ref{sec:DegreeLC} deals with the degree and leading coefficient of these polynomials. A full clarification of why it is sufficient to use only two partitions is given in Section \ref{sec:ConstructionGJP} and \ref{sec:ConstructionXJP}. Finally, the results dealing with the asymptotic behavior of the zeros are given in Section \ref{sec:ZerosResults}. A lower bound on the number of regular zeros is presented in Theorem \ref{thm:XJPN(n)}. The results dealing with the asymptotic behavior of the regular zeros are given in Corollary \ref{cor:XJPRegularZeros} and Theorem \ref{thm:XJPArcsineLaw}. For the exceptional zeros, thasymptotic result is stated in Theorem \ref{thm:XJPExceptionalZeros}. The proofs are listed in Section \ref{sec:ZerosProofs}.


\section{Exceptional Jacobi polynomials in terms of partitions}\label{sec:XJP}
We recall the definition of Jacobi polynomials and elaborate on a few elementary properties. Next, we define the generalized and exceptional Jacobi polynomial as the Wronskian of quasi-rational eigenfunctions of the Jacobi differential operator. The whole section does not contain any new results, most of them are obtained from \cite{Curbera_Duran,Duran-b} and translated to our set-up.

\subsection{Jacobi polynomials}
The Jacobi polynomial has two parameters $\alpha,\beta\in\mathbb{R}$ and is denoted by $P^{(\alpha,\beta)}_{n}$. These polynomials can be defined by Rodrigues' formula \cite[Formula (4.3.1)]{Szego},
\begin{equation}\label{eq:Rodrigues}
	P^{(\alpha,\beta)}_{n}(x) 
		= \frac{(-1)^n}{2^n n!}(1-x)^{-\alpha}(1+x)^{-\beta} \frac{d^n}{dx^n} \left((1-x)^{\alpha+n}(1+x)^{\beta+n}\right).
\end{equation}
If we evaluate the $n^{\text{th}}$ derivative in the Rodrigues' formula, we get an explicit expression for these polynomials,
\begin{equation}\label{eq:JacobiExplicit}
	P^{(\alpha,\beta)}_{n}(x) 
		=\frac{1}{2^{n}} \sum_{j=0}^{n}\binom{n+\alpha}{j}\binom{n+\beta}{n-j}(x-1)^{n-j}(x+1)^{j}.
\end{equation}
This expression shows that $P^{(\alpha,\beta)}_{n}$ is a polynomial for any choice of $\alpha,\beta \in\mathbb{R}$. Contrary to the other classical orthogonal polynomials, the subindex $n$ is not always the degree of the polynomial as a possible degree reduction can occur \cite[Section 4.22 (4.3.1)]{Szego}. To be precise, $\deg P^{(\alpha,\beta)}_{n}=n$ if and only if $\alpha+\beta+n \notin \{-1,-2,\dots,-n\}$. Therefore, one often puts conditions for the parameters $\alpha$ and $\beta$ such that the subindex indicates the degree. 

If the parameters satisfy $\alpha>-1$ and $\beta>-1$, then the Jacobi polynomials are orthogonal on $[-1,1]$ with respect to the positive weight function $(1-x)^{\alpha}(1+x)^{\beta}$. That is,
\begin{equation}\label{eq:JacobiOrthogonal}
	\int_{-1}^{1} P^{(\alpha,\beta)}_{n}(x) P^{(\alpha,\beta)}_{m}(x) (1-x)^{\alpha}(1+x)^{\beta} dx =0,
	\qquad n\neq m.
\end{equation}
As a result, all their zeros are simple and belong to the open interval $(-1, 1)$. More results concerning orthogonality for general parameters can be found in \cite{Kuijlaars_Martinez_Orive}.

\begin{table} 
	\centering
	\begin{tabular}{l }
		$\begin{aligned}[t]
		&\text{Eigenfunction} 
		&& \text{\hspace{2.75cm} Eigenvalue}
		\end{aligned}$ \\
		\hline
		\hline
		$\begin{aligned}[t]
		&P^{(\alpha,\beta)}_{n}(x)
		&&  n(n+\alpha+\beta+1) \\
		&(1+x)^{-\beta} P^{(\alpha,-\beta)}_{n}(x)
		&& n(n+\alpha-\beta+1)-\beta(1+\alpha) \\
		&(1-x)^{-\alpha} P^{(-\alpha,\beta)}_{n}(x)
		&& n(n-\alpha+\beta+1)-\alpha(1+\beta) \\
		&(1+x)^{-\beta} (1-x)^{-\alpha} P^{(-\alpha,-\beta)}_{n}(x)
		&& n(n-\alpha-\beta+1)-(\alpha+\beta)
		\end{aligned}$ \\
		\hline
	\end{tabular}	
	\caption{Quasi-rational eigenfunctions and eigenvalues of the Jacobi operator \eqref{eq:DVJac1}.} 
	\label{tab:1}
\end{table}

The Jacobi polynomials are eigenfunctions of the differential operator
\begin{equation}\label{eq:DVJac1}
	y \mapsto (x^2-1)y''+\left(\alpha-\beta+(\alpha+\beta+2)x\right)y'
\end{equation}
with eigenvalue $n(n+\alpha+\beta+1)$. This operator also has other eigenfunctions which consist of a Jacobi polynomial part, they are listed in Table \ref{tab:1}. In fact, this table consists of all eigenfunctions which are quasi-rational functions, i.e., their log derivative is a polynomial \cite[Section 2.2]{Erdelyi}. We use these eigenfunctions to construct the generalized and exceptional Jacobi polynomial in the following sections. A well-studied feature is that the operator \eqref{eq:DVJac1} can be transformed into an operator which has the following Schr\"odinger form
\begin{equation}\label{eq:DVJac2}
	y \mapsto - y'' + V(x) y,
\end{equation}
i.e., the differential equation of a Sturm-Liouville problem. In our case, the potential $V(x)$ is given by
\begin{equation}\label{eq:Potential}
	V(x)=
		\frac{\left(\alpha-\frac{1}{2}\right)\left(\alpha+\frac{1}{2}\right)}{\sin^2(x)} + \frac{\left(\beta-\frac{1}{2}\right)\left(\beta+\frac{1}{2}\right)}{\cos^2(x)}
\end{equation}
which is (up to a constant) equal to the Darboux-P\"oschl-Teller potential \cite{Grandati_Berard,GomezUllate_Grandati_Milson-L+J,Grandati_Quesne}. It is invariant if we replace 1 or both parameters $\alpha$ and $\beta$ by $-\alpha$ and $-\beta$ respectively. The transformation of \eqref{eq:DVJac1} into \eqref{eq:DVJac2} has to be interpreted as follows, if $y(x)$ is an eigenfunction of \eqref{eq:DVJac1}, then $\sin(x)^{\alpha+\frac{1}{2}} \cos(x)^{\beta+\frac{1}{2}} y(\cos(2x))$ is an eigenfunction of \eqref{eq:DVJac2}. Hence the given eigenfunctions in Table \ref{tab:1} transform to eigenfunctions of \eqref{eq:DVJac2} which are listed in Table \ref{tab:2}. We write the new obtained eigenfunctions as $\varphi^{(\alpha,\beta)}_{n}$, i.e.,
\begin{equation*}
	\varphi^{(\alpha,\beta)}_{n}(x)	=  \sin(x)^{\alpha+\frac{1}{2}} \cos(x)^{\beta+\frac{1}{2}} P^{(\alpha,\beta)}_{n}\left(\cos(2x)\right).
\end{equation*}
Note that all quasi-rational eigenfunctions can be obtained from $\varphi^{(\alpha,\beta)}_{n}$ using the invariance property for the potential, i.e., replacing $(\alpha,\beta)$ by $(\pm\alpha,\pm\beta)$, see Table \ref{tab:2}. This result originates from \cite{Grandati_Berard} where more information about the symmetries can be found.

The transformation of the Jacobi operator into its Schr\"odinger form is important in the context of exceptional polynomials, as it is well-known that exceptional operators and the corresponding exceptional eigenpolynomials \cite{GarciaFerrero_GomezUllate_Milson}, are obtained via a series of Darboux transformations which is directly applicable to an operator in Schr\"odinger form \cite{Crum,Darboux,Sasaki_Tsujimoto_Zhedanov}. 

\begin{table} 
	\centering
	\begin{tabular}{l }
		$\begin{aligned}[t]
			&\text{Eigenfunction} 
			&& \text{\hspace{6.85cm} Eigenvalue}
		\end{aligned}$ \\
		\hline
		\hline
		$\begin{aligned}[t]
		\varphi^{(\alpha,\beta)}_{n}(x)	
			&=  \sin(x)^{\alpha+\frac{1}{2}} \cos(x)^{\beta+\frac{1}{2}} P^{(\alpha,\beta)}_{n}\left(\cos(2x)\right) 
			&&  4\left(n+0.5(\alpha+\beta+1)\right)^{2} \\
		\varphi^{(\alpha,-\beta)}_{n}(x)	
			&=  \sin(x)^{\alpha+\frac{1}{2}} \cos(x)^{-\beta+\frac{1}{2}} P^{(\alpha,-\beta)}_{n}\left(\cos(2x)\right) 
			&& 4\left(n+0.5(\alpha-\beta+1)\right)^{2} \\
		\varphi^{(-\alpha,\beta)}_{n}(x)
			&=  \sin(x)^{-\alpha+\frac{1}{2}} \cos(x)^{\beta+\frac{1}{2}} P^{(-\alpha,\beta)}_{n}\left(\cos(2x)\right)
			&& 4\left(n+0.5(-\alpha+\beta+1)\right)^{2} \\
		\varphi^{(-\alpha,-\beta)}_{n}(x)	
			&=  \sin(x)^{-\alpha+\frac{1}{2}} \cos(x)^{-\beta+\frac{1}{2}} P^{(-\alpha,-\beta)}_{n}\left(\cos(2x)\right)
			&& 4\left(n+0.5(-\alpha-\beta+1)\right)^{2} 
		\end{aligned}$ \\
		\hline
	\end{tabular}	
	\caption{Eigenfunctions and eigenvalues of the operator \eqref{eq:DVJac2}.} 
	\label{tab:2}
\end{table}

\subsection{Generalized Jacobi polynomials}\label{sec:GeneralizedJacobi}
In this section we define the generalized Jacobi polynomial as a Wronskian of eigenfunctions from Table \ref{tab:1}. To end up with a polynomial, we multiply the Wronskian with an appropriate prefactor. We start by defining a Wronskian and a partition.

The Wronskian of a set of sufficiently differentiable functions $f_1,\dots,f_r$ is defined as the determinant of the $r\times r$-matrix $M$ where the entries are $M_{ij}= \frac{d^{i-1}}{dx^{i-1}}f_j$ for $1\leq i,j \leq r$. We write the Wronskian as $\Wr[f_1,\dots,f_r]$. For our purpose, we take the functions $f_i$ equal to the eigenfunctions in Table \ref{tab:1}, see \eqref{eq:fj1}-\eqref{eq:fj2} below.

A partition $\lambda$ of a non-negative integer $N$, denoted by $\lambda \vdash N$, is a weakly decreasing sequence of positive integers $(\lambda_i)_{i=1}^{r}$ such that $|\lambda|:=\sum_{i=1}^{r} \lambda_i=N$. Here, $r$ is called the length of the partition. For each partition $\lambda$, we define a corresponding sequence $n_{\lambda}$ as the strictly decreasing sequence $(n_i)_{i=1}^{r}$ such that $n_i=\lambda_i+r-i$ for $i=1,\dots,r$. Hence $\sum_{i=1}^{r}n_i = |\lambda|+\frac{r(r-1)}{2}$. A partition is called even if $r$ is even and $\lambda_{2i-1}=\lambda_{2i}$ for $i=1,\dots,\frac{r}{2}$.

Now we are able to define the generalized Jacobi polynomial. As mentioned before, this polynomial is defined as the Wronskian of eigenfunctions from Table \ref{tab:1} with an appropriate prefactor. To do this, there are four different kinds of possible sets of eigenfunctions which we can include in the Wronskian. However, it turns out that it is sufficient to take only the first and second type of eigenfunctions in Table \ref{tab:1}. The reasoning behind this is explained in detail in Section \ref{sec:ConstructionGJP} (and was already proven before in \cite{Takemura-Maya}). 

To fix the degrees of the Jacobi polynomials, take two partitions $\lambda$ and $\mu$ of lengths $r_1$ and $r_2$ with corresponding sequences $n_{\lambda}=(n_i)_{i=1}^{r_1}$ and $n_{\mu}=(m_i)_{i=1}^{r_2}$. Set $r=r_1+r_2$ and define
\begin{equation} \label{eq:OmegaLaMu}
	\Omega^{(\alpha,\beta)}_{\lambda, \mu} 
		:=  (1+x)^{(\beta+r_1)r_2} \cdot \Wr\left[f_1, \ldots, f_r \right]
\end{equation}
where 
\begin{align} 
	f_j(x) & = P^{(\alpha,\beta)}_{n_j}(x),  
		&& j=1, \ldots, r_1,  
		\label{eq:fj1} \\
		f_{r_1+j}(x) & = (1+x)^{-\beta}P^{(\alpha,-\beta)}_{m_j}(x), 
		&& j = 1, \ldots, r_2,
		\label{eq:fj2}
\end{align}
are the eigenfunctions as described in Table \ref{tab:1}; the partition $\lambda$ deals with the first type and $\mu$ deals with the second type. Hence, by construction $\Omega^{(\alpha,\beta)}_{\lambda, \mu}$ is the Wronskian of the first two types of eigenfunctions with an appropriate prefactor $(1+x)^{(\beta+r_1)r_2}$ such that we end up with a polynomial. The fact that this prefactor is well-chosen holds in a more general setting (here we work specifically with Jacobi polynomials) and is proven in Proposition \ref{prop:degwr}. 

\begin{definition}\label{def:GJP}
The polynomial $\Omega^{(\alpha,\beta)}_{\lambda, \mu}$ defined in \eqref{eq:OmegaLaMu} is called the generalized Jacobi polynomial of parameters $\alpha$ and $\beta$ associated with partitions $\lambda$ and $\mu$. When both partitions are empty we set $\Omega^{(\alpha,\beta)}_{\emptyset, \emptyset}\equiv 1$.
\end{definition}

\begin{remark}
We defined \eqref{eq:OmegaLaMu} as the generalized Jacobi polynomial as it definitely generalizes the (classical) Jacobi polynomial \eqref{eq:JacobiExplicit}. The special case $\lambda=(n)$ and $\mu=\emptyset$ corresponds to the Jacobi polynomial, i.e., $\Omega^{(\alpha,\beta)}_{(n), \emptyset}=P_n^{(\alpha,\beta)}$. Similarly, if $\lambda=\emptyset$ and $\mu=(m)$ we get $\Omega^{(\alpha,\beta)}_{\emptyset,(m)}=P_m^{(\alpha,-\beta)}$.

The name generalized Jacobi polynomial should not be compared to the Hermite case. For the Hermite case, one often refers to the generalized Hermite polynomial as the Wronskian of Hermite polynomials of consecutive degrees, and not to the Wronskian of an arbitrary (finite) sequence of Hermite polynomials. In our setting, we get that the generalized Jacobi polynomial (for $\mu=\emptyset$) is the Wronskian of an arbitrary (finite) sequence of Jacobi polynomials (with fixed parameters).
\end{remark}

The polynomial \eqref{eq:OmegaLaMu} is defined for every parameters $\alpha$ and $\beta$. Next, we restrict the domain of the parameters such that there is no degree reduction for the Jacobi polynomials and the Wronskian does not vanish.
\begin{itemize}	
	\item[1.] \textbf{No degree reduction} \\
	As discussed before, the degree of a Jacobi polynomial $P_n^{(\alpha,\beta)}$ is not always indicated by its subindex $n$. Therefore we put the following (necessary and sufficient) conditions for the parameters $\alpha$ and $\beta$ such that the subindices $n_i$ and $m_j$ equal the degree of the corresponding Jacobi polynomial in \eqref{eq:fj1} and \eqref{eq:fj2}.
	\begin{equation}\label{Condition2GJP}
		\begin{aligned}
			&\alpha+\beta+n_i \notin \{-1,-2,\dots,-n_i\},	&&\qquad i=1,\dots,r_1, \\
			&\alpha-\beta+m_i \notin \{-1,-2,\dots,-m_i\},	&&\qquad i=1,\dots,r_2. 
		\end{aligned}
	\end{equation}
	
	Observe that if $\alpha+\beta>-1$, then the first $r_1$ conditions in \eqref{Condition2GJP} are satisfied. Similarly, if $\alpha-\beta>-1$, then the last $r_2$ conditions are satisfied.
	
	\item[2.] \textbf{Independent eigenfunctions} \\
	If the conditions \eqref{Condition2GJP} are satisfied, we still need other conditions to determine the degree of the polynomial \eqref{eq:OmegaLaMu} properly. The Wronskian should consist of linearly independent eigenfunctions, otherwise the Wronskian vanishes. Naturally, all eigenfunctions of the same type, i.e., $f_1,\dots,f_{r_1}$ (respectively $f_{r_1+1},\dots,f_r$) are linearly independent as all elements in the sequence $n_{\lambda}$ (respectively $n_{\mu}$) are pairwise different. However, it is possible that eigenfunctions of different types are linearly dependent. For example, when $\beta=0$, an eigenfunction of type \eqref{eq:fj1} coincides with an eigenfunction of type \eqref{eq:fj2} if both polynomials have the same degree. Then, the Wronskian would vanish. In general, for fixed partitions $\lambda$ and $\mu$ and under conditions \eqref{Condition2GJP}, the Wronskian in \eqref{eq:OmegaLaMu} consists of independent functions if and only if
	\begin{align} \label{Condition1GJP}
		&\beta \neq m_j-n_i, && i=1,\dots,r_1 \text{ and } j=1,\dots,r_2.
	\end{align}
	The proof of these sufficiently and necessary conditions follows directly from the property that all functions $f_1,\dots,f_r$ have a different degree in the broad sense. Here, by degree in the broad sense, we mean $\deg\left( (1+x)^{-\beta}P^{(\alpha,-\beta)}_{m}\right) = \deg\left( P^{(\alpha,-\beta)}_{m}\right)- \beta$. 
	
	To end, we have $m_1\geq m_j$ for all $j=1,\dots,r_2$ and $n_i\geq n_{r_1}$ for all $i=1,\dots,r_1$. Hence, the set of conditions \eqref{Condition1GJP} is automatically fulfilled when $\beta> m_1-n_{r_1}$.
\end{itemize}

Now we are able to state the degree and leading coefficient of the generalized Jacobi polynomial where we restrict the parameters as described in \eqref{Condition2GJP} and \eqref{Condition1GJP}. We make use of the Pochhammer symbol $(x)_n=x(x+1)\cdots(x+n-1)$ where $n$ is a non-negative integer and $x\in\mathbb{R}$.

\begin{lemma}\label{lem:deg}
For any partition $\lambda$ and $\mu$, take $\alpha,\beta\in\mathbb{R}$ such that the conditions \eqref{Condition2GJP} and \eqref{Condition1GJP} are satisfied. Then $\Omega^{(\alpha,\beta)}_{\lambda,\mu}$ is a polynomial of degree $|\lambda|+|\mu|$ with leading coefficient
\begin{equation}\label{eq:LeadingCoeff}
	\frac{\prod\limits_{i=1}^{r_1}(n_i+\alpha+\beta+1)_{n_i} \prod\limits_{j=1}^{r_2}(m_j+\alpha-\beta+1)_{m_j}} {2^{\sum\limits_{i=1}^{r_1} n_i + \sum\limits_{j=1}^{r_2}m_j}\prod\limits_{i=1}^{r_1}n_i!\prod\limits_{j=1}^{r_2}m_j!}
	\Delta(n_{\lambda})\Delta(n_{\mu})\prod\limits_{i=1}^{r_1}\prod\limits_{j=1}^{r_2}(m_j-n_i-\beta).
\end{equation}
Here, $\Delta(n_{\lambda})=\prod\limits_{1\leq i < j \leq r_1} (n_j-n_i)$ and $\Delta(n_{\mu})=\prod\limits_{1\leq i < j \leq r_2} (m_j-m_i)$ are the Vandermonde determinants.
\end{lemma}

This leading coefficient can be divided into two parts. Firstly, we have the leading coefficients of each function $f_1,\dots,f_r$ in the Wronskian, for example the leading coefficient of $P^{(\alpha,\beta)}_{n}(x)$ is given by $\frac{(n+\alpha+\beta+1)_{n}}{2^n n!}$ if there is no degree reduction (this follows from \eqref{eq:JacobiExplicit}). Secondly, we obtain two Vandermonde determinants and a product. These three terms can be seen as the Vandermonde determinant of the values $n_1,n_2,\dots,n_{r_1},m_1-\beta,m_2-\beta,\dots,m_{r_2}-\beta$ and they are related to the fact that we are working with a specific Wronskian. The proof of this result is postponed to Section \ref{sec:DegreeLC}, where we prove a more general statement (as it is not necessary to work with Jacobi polynomials), see Proposition \ref{prop:degwr}. The conditions within this proposition transfer to the conditions \eqref{Condition1GJP} in the Jacobi setting. 

\begin{remark}\label{rem:DefinitionDuran}
The result of Lemma \ref{lem:deg} has already been proven by Dur\'an who used a limit procedure of Casorati determinants of Hahn polynomials \cite{Duran-b}. Moreover, he derived a degree statement in a more general setting, see \cite[Lemma 3.3]{Duran-a}. We present a direct approach in Section \ref{sec:DegreeLC}.

It is not immediately clear that our definition of the generalized Jacobi polynomial \eqref{eq:OmegaLaMu} coincides with Dur\'an's definition \cite[Equation (1.7)]{Duran-b}. Dur\'an defined the polynomial as
\begin{equation}\label{eq:Duran}
	\frac{1}{(1+x)^{(r_2-1)r_2}} \left|\tilde{C}^{(\alpha,\beta)}_{\lambda,\mu}\right|
\end{equation}
where $\tilde{C}^{(\alpha,\beta)}_{\lambda,\mu}$ is the $r\times r$-matrix given by
\begin{align*}
	\left(\tilde{C}^{(\alpha,\beta)}_{\lambda,\mu}\right)_{i,j}
		&=(-1)^{i-1} \left(P^{(\alpha,\beta)}_{n_j}\right)^{(i-1)}(x) &&i=1,\dots,r \text{ and } j=1,\dots,r_1, \\
	\left(\tilde{C}^{(\alpha,\beta)}_{\lambda,\mu}\right)_{i,r_1+j}
		&=(\beta-m_j)_{i-1} (1+x)^{r-i} P^{(\alpha+i-1,-\beta-i+1)}_{m_j}(x) &&i=1,\dots,r \text{ and } j=1,\dots,r_2.
\end{align*}
Next, we clarify that both definitions are the same (up to a possible sign). 

Recall our definition of $\Omega^{(\alpha,\beta)}_{\lambda, \mu}$ in \eqref{eq:OmegaLaMu}. The idea is to write this expression for the generalized Jacobi polynomial as
\begin{equation}\label{eq:DefDur2}
	\Omega^{(\alpha,\beta)}_{\lambda, \mu} 
		=  \frac{1}{(1+x)^{(r_2-1)r_2}} \left|C^{(\alpha,\beta)}_{\lambda,\mu}\right|
\end{equation}
where $C^{(\alpha,\beta)}_{\lambda,\mu}$ is an $r\times r$-matrix depending on the parameters $\alpha,\beta$ and the partitions $\lambda,\mu$. We do not write $\tilde{C}^{(\alpha,\beta)}_{\lambda,\mu}$ in \eqref{eq:DefDur2} as our definitions are not completely the same.

We start by stating a general derivative expression for our functions $f_,\dots,f_r$ as defined in \eqref{eq:fj1} and \eqref{eq:fj2}. Both identities follow simply from Rodrigues' formula \eqref{eq:Rodrigues}. For all $k\in\mathbb{N}$,
\begin{align}
	\frac{d^k}{dx^k} P^{(\alpha,\beta)}_{n}(x) 
		&= \frac{(n+\alpha+\beta+1)_{k}}{2^k} P^{(\alpha+k,\beta+k)}_{n-k}(x), 
	\label{eq:JacobiDerivative}\\
	\frac{d^k}{dx^k}\left((1+x)^{-\beta}P^{(\alpha,-\beta)}_{m}(x)\right)
		&= (m-\beta-k+1)_{k} (1+x)^{-\beta-k}P^{(\alpha+k,-\beta-k)}_{m}(x),
	\nonumber
\end{align}
where we used the Pochhammer symbol and we set $P^{(\alpha,\beta)}_{-N}\equiv 0$ if $N>0$. Using both results, we express the Wronskian in \eqref{eq:OmegaLaMu} as a determinant of the $r\times r$-matrix where the entries of the matrix are given by the above derivatives. If we multiply the last $r_2$ columns of this matrix with the factor $(1+x)^{\beta+r-1}$, we obtain
\begin{equation}\label{eq:DefDur3}
	\Wr[f_1,\dots,f_r]
		= (1+x)^{-r_2(\beta+r-1)}
		\begin{vmatrix}
			C^{(\alpha,\beta)}_{\lambda,\mu}
		\end{vmatrix}
\end{equation}
where $C^{(\alpha,\beta)}_{\lambda,\mu}$ is an $r\times r$-matrix which can be expressed as two blocks, 
\begin{equation*}
	C^{(\alpha,\beta)}_{\lambda,\mu}
		= \begin{pmatrix}
			A^{(\alpha,\beta)}_{\lambda} & B^{(\alpha,\beta)}_{\mu}
		\end{pmatrix}.
\end{equation*}	
The matrix $A^{(\alpha,\beta)}_{\lambda}$ is an $r\times r_1$-block and $B^{(\alpha,\beta)}_{\mu}$ is an $r\times r_2$-block, they are
\begin{align*}
	&A^{(\alpha,\beta)}_{\lambda}
		= \begin{pmatrix}
			P^{(\alpha,\beta)}_{n_1}(x) &\dots &P^{(\alpha,\beta)}_{n_{r_1}}(x) 
			\vspace{0.25cm}\\
			c_{n_1,1} P^{(\alpha+1,\beta+1)}_{n_1}(x) &\dots & c_{n_{r_1},1} P^{(\alpha+1,\beta+1)}_{n_{r_1}}(x)
			\vspace{0.25cm}\\
			\vdots &\ddots &\vdots 
			\vspace{0.25cm}\\
			c_{n_1,r-1} P^{(\alpha+r-1,\beta+r-1)}_{n_1}(x) &\dots & c_{n_{r_1},r-1} P^{(\alpha+r-1,\beta+r-1)}_{n_{r_1}}(x)
		\end{pmatrix} \\
	&B^{(\alpha,\beta)}_{\mu}
		= \begin{pmatrix}
			(1+x)^{r-1} P^{(\alpha,-\beta)}_{m_1}(x) & \dots & (1+x)^{r-1} P^{(\alpha,-\beta)}_{m_{r_2}}(x)
			\vspace{0.25cm} \\
			d_{m_1,1} (1+x)^{r-2} P^{(\alpha+1,-\beta-1)}_{m_1}(x) & \dots & d_{m_{r_2},1} (1+x)^{r-2} P^{(\alpha+1,-\beta-1)}_{m_{r_2}}(x)
			\vspace{0.25cm} \\
			\vdots &\ddots & \vdots
			\vspace{0.25cm} \\
			d_{m_1,r-1} P^{(\alpha+r-1,-\beta-r+1)}_{m_1}(x) & \dots & d_{m_{r_2},r-1} P^{(\alpha+r-1,-\beta-r+1)}_{m_{r_2}}(x)
		\end{pmatrix}	
\end{align*} 
with
\begin{align*}
	& c_{n_j,i}= \frac{(n_j+\alpha+\beta+1)_j}{2^i}, &&j=1,\dots,r_1, \quad i=1,\dots,r-1,\\
	& d_{m_j,i}= (m_j-\beta-i+1)_i, &&j=1,\dots,r_2, \quad i=1,\dots,r-1.
\end{align*}
Combining \eqref{eq:OmegaLaMu} and \eqref{eq:DefDur3}, we obtain \eqref{eq:DefDur2}. Finally, the expressions \eqref{eq:Duran} and the right hand side of \eqref{eq:DefDur2} are the same up to a possible sign. The difference in sign is due to a different choice for the factors $d_{m_j,i}$. Dur\'an prefers to work with $(\beta-m_j)_{i-1}$ instead of our choice of $d_{m_j,i}$ and as a result our definitions differ by a factor $(-1)^{\lfloor \frac{r}{2}\rfloor}$, where we used the floor function, which does not have any further consequences. 
\end{remark}

\begin{remark}\label{rem:type2and3}
The generalized Jacobi polynomial \eqref{eq:OmegaLaMu} is defined using the eigenfunctions $P^{(\alpha,\beta)}_{n_j}$, see \eqref{eq:fj1}, and $(1+x)^{-\beta}P^{(\alpha,-\beta)}_{m_{j}}$, see \eqref{eq:fj2}. As we explain in Section \ref{sec:ConstructionGJP}, when we use all four types of eigenfunctions in Table \ref{tab:1}, we can always reduce to \eqref{eq:OmegaLaMu}. Nevertheless, our choice in \eqref{eq:OmegaLaMu} seems a bit arbitrary; it is also reasonable to define the generalized Jacobi polynomial using  $P^{(\alpha,\beta)}_{n_j}$ and $(1-x)^{-\alpha}P^{(-\alpha,\beta)}_{m_j}$, i.e.,
\begin{equation}\label{eq:OmegaTilde}
	\widetilde{\Omega}^{(\alpha,\beta)}_{\lambda,\mu} 
		:= (1-x)^{(\alpha+r_1)r_2} \cdot \Wr\left[\tilde{f}_1, \ldots, \tilde{f}_r \right]
\end{equation}
where 
\begin{align} 
	\tilde{f}_j(x) & = P^{(\alpha,\beta)}_{n_j}(x),  
	&& j=1, \ldots, r_1,  
	\label{eq:ftilde1}\\
	\tilde{f}_{r_1+j}(x) & = (1-x)^{-\alpha}P^{(-\alpha,\beta)}_{m_j}(x), 
	&& j = 1, \ldots, r_2.
	\label{eq:ftilde2}
\end{align}
In this definition, the parameter $\beta$ is fixed, where in \eqref{eq:OmegaLaMu} we chose to fix the parameter $\alpha$. However, both polynomials $\Omega^{(\alpha,\beta)}_{\lambda,\mu}$ and $\widetilde{\Omega}^{(\alpha,\beta)}_{\lambda,\mu}$ coincide in the following sense,
\begin{equation}\label{eq:Omega-x}
	\Omega^{(\alpha,\beta)}_{\lambda,\mu}(-x)
		= (-1)^{|\lambda|+|\mu|+r_1r_2} \widetilde{\Omega}^{(\beta,\alpha)}_{\lambda,\mu}(x),
		\qquad 
		x\in\mathbb{C}.
\end{equation}
Hence it does not matter which definition \eqref{eq:OmegaLaMu} or \eqref{eq:OmegaTilde} we choose, all results can easily be translated to the other case. Our choice \eqref{eq:OmegaLaMu} was made such that our definition coincides with Dur\'an's choice \cite{Duran-b}, in fact, Dur\'an's choice was arbitrary too. To prove identity \eqref{eq:Omega-x}, consider the general Wronskian property
\begin{equation}\label{eq:Wr2} 
	\Wr[g_1 \circ h,\dots,g_r\circ h](x) 
		= \left( h'(x)\right)^{\frac{r(r-1)}{2}} \cdot \Wr[g_1,\dots,g_r](h(x))
\end{equation}
which holds for sufficiently many differentiable functions $g_1,\dots,g_r,h$. If we apply this result on the Wronskian in \eqref{eq:OmegaLaMu}, where we take $g_i=f_i$ and $h(x)=-x$, we find
\begin{equation}\label{eq:OmegaTildeProof1}
	\Omega^{(\alpha,\beta)}_{\lambda,\mu}(-x)
		= (-1)^{\frac{r(r-1)}{2}} (1-x)^{(\beta+r_1)r_2} \Wr\left[f_1(-x),\dots,f_r(-x)\right]
\end{equation}
where $f_1,\dots,f_r$ are defined in \eqref{eq:fj1} and \eqref{eq:fj2}. Next, the Jacobi polynomials satisfy
\begin{equation}\label{eq:Jacobi-x}
	P^{(\alpha,\beta)}_{n}(-x) 
		= (-1)^n P^{(\beta,\alpha)}_{n}(x)
\end{equation}
which follows directly from the Rodrigues' formula \eqref{eq:Rodrigues}. Hence if we plug \eqref{eq:Jacobi-x} into the Wronskian of \eqref{eq:OmegaTildeProof1}, we get
\begin{equation}\label{eq:OmegaTildeProof2}
	\Omega^{(\alpha,\beta)}_{\lambda,\mu}(-x)
		= (-1)^{\frac{r(r-1)}{2}+\sum\limits_{i=1}^{r_1}n_i + \sum\limits_{i=1}^{r_2} m_i} (1-x)^{(\beta+r_1)r_2} \Wr\left[\tilde{f}_1(x), \ldots, \tilde{f}_r(x) \right]
\end{equation}
where $\tilde{f}_1,\dots,\tilde{f}_r$ are defined in \eqref{eq:ftilde1}-\eqref{eq:ftilde2}. Finally, \eqref{eq:OmegaTildeProof2} simplifies to \eqref{eq:Omega-x}.
\end{remark}

Further on,  we discuss a couple of known results for generalized Jacobi polynomials where we translate them to our partition notation. We denote $\widehat{\Omega}^{(\alpha,\beta)}_{\lambda,\mu}$ as the monic generalized Jacobi polynomial, i.e., $k_{\lambda,\mu} \widehat{\Omega}^{(\alpha,\beta)}_{\lambda,\mu} = \Omega^{(\alpha,\beta)}_{\lambda,\mu}$ where $k_{\lambda,\mu}$ is given by \eqref{eq:LeadingCoeff}. 

\paragraph{Interchanging the partitions.} 
The generalized Jacobi polynomial is defined via two partitions $\lambda$ and $\mu$. When we interchange both partitions, it turns out that it corresponds to changing the sign of the parameter $\beta$ and a possible sign change depending on the length of both partitions.

\begin{lemma}\label{lem:lemma4}
For any $\alpha,\beta\in\mathbb{R}$ and for any partitions $\lambda$ and $\mu$ we have
\begin{equation}\label{eq:duality}
	\Omega^{(\alpha,\beta)}_{\lambda,\mu} = (-1)^{r_1 r_2} \Omega^{(\alpha,-\beta)}_{\mu,\lambda}.
\end{equation}
\end{lemma}
\begin{proof}
We start with the well-known Wronskian property that for sufficiently differentiable functions $g_1,\dots,g_r,h$, we have
\begin{equation}\label{eq:Wr1}
	\Wr[ h \cdot g_1,\dots, h \cdot g_r] 
		= \left(h(x)\right)^{r} \cdot \Wr[g_1, \dots, g_r].
\end{equation}
Using this result for $h(x)=(1+x)^{-\beta}$ and $g_i=f_i$ for $i=1,\dots,r$ where $f_i$ is defined in \eqref{eq:fj1} and \eqref{eq:fj2}, we obtain
\begin{equation*}
	\Omega^{(\alpha,\beta)}_{\lambda,\mu}(x) 
		= (1+x)^{(r_1+\beta)r_2 -\beta r} \Wr \left[(1+x)^{\beta}f_1,\dots,(1+x)^{\beta}f_r\right].
\end{equation*}
The functions in the Wronskian are given by
\begin{align*}
	&(1+x)^{\beta}f_{j}(x) =  (1+x)^{\beta} P^{(\alpha,\beta)}_{n_j}(x), &&j=1,\dots,r_1, \\
	&(1+x)^{\beta}f_{r_1+j}(x) = P^{(\alpha,-\beta)}_{m_j}(x), &&j=1,\dots,r_2.
\end{align*}
Permuting the first $r_1$ functions with the last $r_2$ functions gives 
\begin{equation}\label{eq:proofduality}
	\Omega^{(\alpha,\beta)}_{\lambda,\mu}(x) 
		= (-1)^{r_1r_2}(1+x)^{(r_2-\beta)r_1} \Wr \left[\tilde{f}_1,\dots,\tilde{f}_r\right]
\end{equation}
where
\begin{align*}	
	&\tilde{f}_{j}(x) = P^{(\alpha,-\beta)}_{m_j}(x), &&j=1,\dots,r_2, \\
	&\tilde{f}_{r_2+j}(x) =  (1+x)^{\beta} P^{(\alpha,\beta)}_{n_j}(x), &&j=1,\dots,r_1.
\end{align*}
If we look at \eqref{eq:proofduality}, we see that, up to a possible sign, the right hand side is $\Omega^{(\alpha,-\beta)}_{\mu,\lambda}$ which establish identity \eqref{eq:duality}. 
\end{proof}

\paragraph{The conjugated partitions.}
Every partition $\lambda$ has a conjugated partition $\lambda'$. This partition $\lambda'$ is defined as the weakly decreasing sequence $(\lambda'_i)_{i=1}^{\lambda_1}$ where $\lambda'_i=\#\{j \mid \lambda_j\geq i\}$. One could ask the question what would happen if one replaces $\lambda$ and $\mu$ by its conjugate partition $\lambda'$ and $\mu'$ in \eqref{eq:OmegaLaMu}. Curbera and Dur\'an showed that both monic generalized Jacobi polynomials coincide if we transpose both parameters in a specific way.

\begin{lemma}[Theorem 8.1 in \cite{Curbera_Duran}]\label{lem:ConPar}
For any partitions $\lambda$ and $\mu$, let $\lambda'$ and $\mu'$ denote their conjugated partitions. Set 
\begin{align*}
	&s_1 = \lambda_1+\mu_1+r_1+r_2,
	&&s_2 = \lambda_1-\mu_1+r_1-r_2,
\end{align*}
with the convention that $\lambda_1=0$ if $r_1=0$ and $\mu_1=0$ if $r_2=0$. Take $\alpha,\beta\in\mathbb{R}$ in such a way that the conditions \eqref{Condition2GJP} and \eqref{Condition1GJP} are satisfied in terms of the partitions $\lambda,\mu$ and parameters $\alpha,\beta$, but also in terms of the partitions $\lambda',\mu'$ and parameters $-\alpha-s_1,-\beta-s_2$. 
Then
\begin{equation}\label{eq:ConjugatedPartitions}
	\widehat{\Omega}^{(\alpha,\beta)}_{\lambda,\mu} 
	= \widehat{\Omega}^{(-\alpha-s_1,-\beta-s_2)}_{\lambda',\mu'}.
\end{equation}
\end{lemma}

The proof in \cite{Curbera_Duran} is based on a limit procedure of Casorati determinants of Hahn polynomials. As we discuss a more general setting in Section \ref{sec:ConstructionGJP}, Lemma \ref{lem:ConPar} is a special case of Theorem \ref{thm:ShiftMaya}. So we obtain an alternative proof of the above lemma which is presented in Section \ref{sec:SpecialCase}. 

\paragraph{The orthogonality region $[-1,1]$.}
In the next section we define exceptional Jacobi polynomials. For specific partitions and parameters, these polynomials are orthogonal on the interval $[-1,1]$. Their corresponding weight function consists of the generalized Jacobi polynomial squared in the denominator, see \eqref{eq:Weight}. So the zeros of this polynomial are poles of the weight function. Hence we have to classify when the generalized Jacobi polynomial has no zeros in $(-1,1)$. If $r_1=0$, we set $m_1=0$. Recall that the conditions \eqref{Condition1GJP} are automatically fulfilled when $\beta > m_1$. 

\begin{lemma}\label{lem:OmegaZerosOrth}
For any partitions $\lambda$ and $\mu$, take $\alpha>-1$ and $\beta > m_1$ such that the conditions \eqref{Condition2GJP} are satisfied. If $\lambda$ is an even partition, then $\Omega^{(\alpha,\beta)}_{\lambda,\mu}(x)\neq0$ for all $x\in[-1,1]$.
\end{lemma}

The conditions in Lemma \ref{lem:OmegaZerosOrth} can be relaxed, as discussed in \cite{Duran-b}. In that paper, necessary conditions are derived. They seem necessary too, however, a technical extra assumption is needed, see \cite[Theorem 6.5]{Duran-b}. As these conditions are rather complex, we specified to stronger assumptions in Lemma \ref{lem:OmegaZerosOrth} which are easy to state in terms of partitions. 

The conditions in \cite{Duran-b} are obtained via a limit procedure. As suggested by a referee, it might be possible to rediscover these results using an extension of the Krein-Adler theorem, see for example \cite{GarciaFerrero_GomezUllate}, so that the extra technical assumption can be avoided. We omit further research in this paper.

\paragraph{Value at the endpoints of the orthogonality region.}
In Corollary \ref{cor:XJPRegularZeros}, we need the conditions $\Omega^{(\alpha,\beta)}_{\lambda,\mu}(\pm1) \neq 0$, i.e., the generalized Jacobi polynomial does not vanish at the endpoints of the orthogonality region. An explicit expression for these values can be found in \cite[Lemma 5.1]{Duran-b}. From this result, we easily derive when these values are non-zero. We set $n_1=0$ if $r_1=0$ and similarly $m_1=0$ if $r_2=0$.

\begin{lemma}\label{lem:Omega+-1}
For any partition $\lambda$ and $\mu$, take $\alpha,\beta\in\mathbb{R}$ such that the conditions \eqref{Condition2GJP} and \eqref{Condition1GJP} are satisfied. 
\begin{itemize}
\item[(a)]
If $\alpha \notin \{ -1,-2,\dots,-\max\{n_1,m_1\}\}$ and
\begin{align*}
	&\alpha \neq -n_i-m_j-1,	&& i=1,\dots,r_1 \text{ and } j=1,\dots,r_2,
\end{align*}
then $\Omega^{(\alpha,\beta)}_{\lambda,\mu}(1) \neq 0$.

\item[(b)]
If $\beta \notin \{-n_1,-n_1+1,\dots,-1,0,1,\dots,m_1-1,m_1\}$, then $\Omega^{(\alpha,\beta)}_{\lambda,\mu}(-1) \neq 0$. \\
Moreover, if $r_1=0$ or $r_2=0$, then $\Omega^{(\alpha,\beta)}_{\lambda,\mu}(-1) \neq 0$ for $\beta=0$.
\end{itemize}
\end{lemma}

The conditions in Lemma \ref{lem:Omega+-1} are sufficient but it is possible to improve the set of forbidden values for $\alpha$ and $\beta$. We omit further specification.

\subsection{Exceptional Jacobi polynomials}\label{sec:XJPDef}
The exceptional Jacobi polynomial is defined similar to the generalized Jacobi polynomial \eqref{eq:OmegaLaMu}. The only difference is that we add a Jacobi polynomial in the Wronskian of \eqref{eq:OmegaLaMu} and change the exponent of the prefactor such that we again end up with a polynomial.

Fix the partitions $\lambda$ and $\mu$ and let $n_{\lambda}$ and $n_{\mu}$ denote their corresponding sequences. Let $r_1$ (respectively $r_2$) be the length of the partition $\lambda$ (respectively $\mu$). Set $r=r_1+r_2$ and suppose that the functions $f_1,\dots,f_r$ are as in \eqref{eq:fj1} and \eqref{eq:fj2}. Finally, fix the parameters $\alpha$ and $\beta$ and consider the function
\begin{equation} \label{eq:XJPexample}
	(1+x)^{(\beta+r_1+1)r_2} \cdot \Wr \left[f_1, \ldots, f_r, P_s^{(\alpha,\beta)} \right], 
	\qquad s\geq 0.
\end{equation}
Then, up to a possible sign change, this polynomial is the generalized Jacobi polynomial for the partitions $\tilde{\lambda}$ and $\mu$. Here we denote $\tilde{\lambda}$ as the partition where the corresponding sequence consists of the elements $n_1,\dots,n_{r_1},s$. The idea is to vary $s$ in such a way that the above polynomial \eqref{eq:XJPexample} exists. That is, $s\geq0$ and $s\neq n_j$ for $j=1,\dots,r_1$. The obtained polynomials will be the exceptional Jacobi polynomials corresponding to the partitions $\lambda$ and $\mu$. For future purposes, we set $s = n - |\lambda|- |\mu| + r_1$. Hence we can vary $n$ instead of $s$. We obtain the following conditions for $n$.

\begin{definition}\label{def:Nlamu}
The degree sequence associated with partitions $\lambda$ and $\mu$ is
\begin{equation} \label{eq:Nlamu} 
	\mathbb{N}_{\lambda, \mu}
		:= \{ n \in \mathbb{N}_0 | n\geq |\lambda|+|\mu| - r_1 \text{ and } n - |\lambda|-|\mu| \neq \lambda_j-j \text{ for } j=1,\dots,r_1 \}.
\end{equation}
\end{definition}

The first condition of \eqref{eq:Nlamu} ensures that $s$ is a non-negative integer while the other conditions give that $s\neq n_j$ for $j=1,\dots,r_1$. Now we are able to define the exceptional Jacobi polynomials.

\begin{definition}
The exceptional Jacobi polynomials of parameters $\alpha$ and $\beta$ associated with partitions $\lambda$ and $\mu$ are given by
\begin{equation}\label{def:XJP1}
	P^{(\alpha,\beta)}_{\lambda,\mu,n}(x) 
		:=  (1+x)^{(\beta+r_1+1)r_2} \cdot \Wr \left[ f_1, \ldots, f_r, P^{(\alpha,\beta)}_s \right], 
	\qquad n \in \mathbb N_{\lambda, \mu},
\end{equation}
where $s = n - |\lambda|- |\mu| + r_1$ and $f_1, \ldots, f_r$ are as in \eqref{eq:fj1} and \eqref{eq:fj2}.
\end{definition}

\begin{remark}
If both partitions are empty, the degree sequence, defined in \eqref{eq:Nlamu}, equals the non-negative integers and the exceptional Jacobi polynomials simplify to the (classical) Jacobi polynomials, i.e., $P^{(\alpha,\beta)}_{\emptyset,\emptyset,n}\equiv P^{(\alpha,\beta)}_{n}$. 
\end{remark}

As for the generalized Jacobi polynomial, we put conditions for the parameters $\alpha$ and $\beta$ such that we can define the degree of the exceptional Jacobi polynomial properly. Recall that the exceptional Jacobi polynomial is a specific generalized Jacobi polynomial, hence the conditions \eqref{Condition2GJP} and \eqref{Condition1GJP} transfer to similar conditions for the exceptional Jacobi polynomials. Fix $n\in\mathbb{N}_{\lambda,\mu}$ and set $s = n - |\lambda|- |\mu| + r_1$.

\begin{itemize}
	\item[1.] \textbf{No degree reduction} \\
	We exclude degree reduction for the Jacobi polynomials in the Wronskian of \eqref{def:XJP1}. Hence, the conditions
	\begin{equation}\label{Condition2XJP} 
	\begin{aligned}
			&\alpha+\beta+n_i \notin \{-1,-2,\dots,-n_i\},	&&\qquad i=1,\dots,r_1,
			\\
			&\alpha+\beta+s   \notin \{-1,-2,\dots, -s \}, 
			\\
			&\alpha-\beta+m_i \notin \{-1,-2,\dots,-m_i\},	&&\qquad i=1,\dots,r_2.
	\end{aligned}
	\end{equation}
	are sufficient to obtain that the degree of all Jacobi polynomials is indicated by their subindex. The value of $\alpha+\beta+s$ is always positive for sufficiently large values of $n$. Hence, if $n$ tends to infinity, then the set of all conditions in \eqref{Condition2XJP} reduces to \eqref{Condition2GJP}. For later purposes, see Lemma \ref{lem:JacobiShiftParameter} and Lemma \ref{lem:XJPlincomJac} below, we extend this set of conditions to 
	\begin{equation}\label{Condition2XJPbis} 
		\begin{aligned}
			&\alpha+\beta+n_i \notin \{-1,-2,\dots,-n_i\},	&&\qquad i=1,\dots,r_1,
			\\
			&\alpha+\beta+s   \notin \{-1,-2,\dots, -s-2r \}, 
			\\
			&\alpha-\beta+m_i \notin \{-1,-2,\dots,-m_i\},	&&\qquad i=1,\dots,r_2.
		\end{aligned}
	\end{equation}
	
	\item[2.] \textbf{Independent eigenfunctions} \\
	To obtain that the Wronskian in \eqref{def:XJP1} would not vanish because of linearly dependent functions, we have the following necessary and sufficient conditions,
	\begin{equation}\label{Condition1XJP}
	\begin{aligned} 
			&\beta \neq m_j-n_i, \qquad &&i=1,\dots,r_1 \text{ and } j=1,\dots,r_2, \\
			&\beta \neq m_j-s,   \qquad &&j=1,\dots,r_2,
	\end{aligned} 
	\end{equation}
	where we assume the conditions \eqref{Condition2XJP} to be satisfied. The second set of conditions in \eqref{Condition1XJP} is satisfied when $n$ is sufficiently large, then the conditions \eqref{Condition1XJP} reduce to \eqref{Condition1GJP}.
\end{itemize}

Assuming the above conditions for the parameters, we apply Lemma \ref{lem:deg} to determine the degree of the polynomial in \eqref{def:XJP1}.

\begin{lemma}\label{lem:XJPdeg}
For any partitions $\lambda$ and $\mu$, take $\alpha,\beta\in\mathbb{R}$ such that the conditions \eqref{Condition2XJP} and \eqref{Condition1XJP} are satisfied. Then we have $\deg P^{(\alpha,\beta)}_{\lambda,\mu,n} = n$ for every $n\in\mathbb{N}_{\lambda,\mu}$.
\end{lemma}

This lemma tells us that under specified conditions, the exceptional Jacobi polynomial of degree $n$ exists for all non-negative integers $n$ except for finitely many. There are $|\lambda|+|\mu|$ exceptional degrees, the degrees which are not reached. Stated otherwise, we have a set of polynomials for a cofinite number of degrees. The leading coefficient of $P^{(\alpha,\beta)}_{\lambda,\mu,n}$ can be determined as in Lemma \ref{lem:deg}.

\begin{remark}
The main ingredient to obtain the exceptional Jacobi polynomial \eqref{def:XJP1} is to add an eigenfunction in the Wronskian of \eqref{eq:OmegaLaMu}. However, one could also define the exceptional polynomial by adding the other type of eigenfunction into the Wronskian, we would have
\begin{equation}\label{def:XJP2}
	\widetilde{P}^{(\alpha,\beta)}_{\lambda,\mu,n}(x) 
		=(1+x)^{(\beta+r_1)(r_2+1)} \cdot \Wr\left[f_1, \ldots, f_r, (1+x)^{-\beta} P^{(\alpha,-\beta)}_s \right]
\end{equation}
where we take $s = n - |\lambda| - |\mu| + r_2 \geq 0$. This polynomial is comparable to our definition in \eqref{def:XJP1} as by the duality concept stated in Lemma \ref{lem:lemma4}, we get 
\begin{equation*}
	\widetilde{P}^{(\alpha,\beta)}_{\lambda,\mu,n}
	= (-1)^{r_1 r_2} P^{(\alpha,-\beta)}_{\mu,\lambda,n}, 
	\qquad
	n\in\mathbb{N}_{\mu,\lambda}.
\end{equation*}
Hence, our definition \eqref{def:XJP1} also captures this situation. More general details concerning this aspect are discussed in Section \ref{sec:XJPrevisited}.
\end{remark}

To end this section, we elaborate on the orthogonality and completeness of these polynomials in terms of partitions. A far more detailed study about this topic is presented in \cite{Duran-b}.

\paragraph{Orthogonality and completeness.}
Jacobi polynomials form a complete set of orthogonal polynomials on the interval $[-1,1]$ when $\alpha>-1$ and $\beta>-1$, see \eqref{eq:JacobiOrthogonal}. Dur\'an derived a similar result for exceptional Jacobi polynomials \cite{Duran-b}. He showed that under explicit conditions, the exceptional Jacobi polynomials are orthogonal for the weight function
\begin{equation}\label{eq:Weight}
	W^{(\alpha,\beta)}_{\lambda, \mu}(x) = \frac{(1-x)^{\alpha+r_1+r_2}(1+x)^{\beta+r_1-r_2}}
		{\left(\Omega^{(\alpha,\beta)}_{\lambda, \mu}(x) \right)^2}, \qquad x\in(-1,1).
\end{equation}
We choose to present stronger assumptions than Dur\'an derived. Moreover, let $m_1=0$ if $r_2=0$, and observe that the conditions \eqref{Condition1XJP} are automatically fulfilled when $\beta > m_1$.

\begin{lemma}\label{lem:XJPorth}
For any partitions $\lambda$ and $\mu$, take $\alpha>-1$ and $\beta > m_1$ such that the conditions \eqref{Condition2XJP} are satisfied for all $n\in\mathbb{N_{\lambda,\mu}}$. If $\lambda$ is an even partition, then the polynomials $P^{(\alpha,\beta)}_{\lambda,\mu,n}$ for $n \in \mathbb N_{\lambda, \mu}$ are orthogonal on $[-1,1]$ with respect to the positive weight function $W^{(\alpha,\beta)}_{\lambda, \mu}$ as defined in \eqref{eq:Weight}. That is, if $n, m \in \mathbb N_{\lambda, \mu}$, then
\begin{equation*}
	\int_{-1}^{1} P^{(\alpha,\beta)}_{\lambda, \mu, n}(x) P^{(\alpha,\beta)}_{\lambda, \mu, m}(x) 
		W^{(\alpha,\beta)}_{\lambda, \mu}(x) dx = 0, \qquad n\neq m.
\end{equation*}
Moreover, they form a complete set in $L^2\left([-1,1],W^{(\alpha,\beta)}_{\lambda, \mu}dx\right)$.
\end{lemma}

Imposing the conditions of Lemma \ref{lem:XJPorth}, the weight function \eqref{eq:Weight} has no poles in $(-1,1)$. This follows from Lemma \ref{lem:OmegaZerosOrth} as it ensures that the denominator has no zeros, and from the straightforward calculation $\alpha+r_1+r_2\geq\alpha>-1$ and $\beta+r_1-r_2> m_1-r_2\geq 0$, we get that the numerator does not give rise to any poles. More general conditions can be found in \cite{Duran-b}.

\paragraph{}
To summarize, the set of exceptional Jacobi polynomials $P^{(\alpha,\beta)}_{\lambda, \mu, n}$ for $n\in\mathbb{N}_{\lambda,\mu}$ is properly defined for all partitions and parameters. If they are well-chosen, the degree of $P^{(\alpha,\beta)}_{\lambda, \mu, n}$ equals $n$, see Lemma \ref{lem:XJPdeg}. For even stronger conditions, the set is a complete orthogonal system as indicated in Lemma \ref{lem:XJPorth}. In the rest of this paper, we discuss the set-up of the Wronskian in more detail and tackle the asymptotic behavior of the zeros of exceptional Jacobi polynomials. For each result, we specify which conditions for partitions and parameters are required.

\begin{remark}
Exceptional orthogonal polynomials are divided into three classes, exceptional Hermite, exceptional Laguerre and exceptional Jacobi polynomials \cite{GarciaFerrero_GomezUllate_Milson}. Each class can be defined using partitions. For exceptional Jacobi polynomials, we need two partitions to define these polynomials (as we need two sets of eigenfunctions). This coincides with the Laguerre case \cite{Bonneux_Kuijlaars} while for the Hermite case \cite{Kuijlaars_Milson} we only need 1 partition.
\end{remark}

\begin{remark}
The functions $f_1,\dots,f_r$ in \eqref{def:XJP1} are often called the seed functions as they are used in the Darboux transformations to obtain the exceptional Jacobi polynomial. In our setting, the seed functions are eigenfunctions as indicated in Table \ref{tab:1}. One could wonder if there are other seed functions to construct exceptional Jacobi polynomials. The answer is yes, one can use generalized eigenfunctions, i.e., for an operator $T$, a function $\phi$ is called a generalized eigenfunction if $(T-\lambda)(\phi) \neq 0$ and $(T-\lambda)^2(\phi) = 0$. An example of such a construction can be found in \cite{Bagchi_Grandati_Quesne}. So, the list of exceptional Jacobi polynomials in this paper, indexed by two partitions, is not complete: it does not exhaust all possibilities. Although, the feature of generalized eigenfunctions somehow escapes the normal setting and therefore can be seen as a degenerated case. More details concerning this feature can be found in \cite{Grandati_Quesne} and further study is definitely needed to obtain a full classification.
\end{remark}

\section{Degree and leading coefficient of $\Omega^{(\alpha,\beta)}_{\lambda, \mu}$}\label{sec:DegreeLC}
We defined the generalized Jacobi polynomial as a Wronskian with an appropriate prefactor, see \eqref{eq:OmegaLaMu}. Similarly, the exceptional Jacobi polynomial \eqref{def:XJP1} is just a specific generalized Jacobi polynomial as described in the beginning of Section \ref{sec:XJPDef}. It is not directly clear from the definition that these functions are polynomials. This section deals with this issue. We prove that the generalized Jacobi polynomial is a polynomial and derive its degree and leading coefficient. The result we obtain holds for general polynomials, although we only applied it in the Jacobi setting as specified in Lemma \ref{lem:deg} and Lemma \ref{lem:XJPdeg}. The following proposition is the main result of this section.

\begin{proposition}\label{prop:degwr}
Let $r_1,r_2$ be non-negative integers and set $r=r_1+r_2$. Let $R_1,\dots,R_r$ be non-zero monic polynomials such that $\deg R_i \neq \deg R_{j}$ for $i\neq j$ and $1\leq i,j \leq r_1$ or $r_1+1 \leq i,j \leq r$. Take $c,\beta\in\mathbb{R}$. Then the function
\begin{equation}\label{eq:OmegaDef}
	\Omega := (x+c)^{(r_1-\beta)r_2} \cdot \Wr\left[R_1,\dots,R_{r_1},(x+c)^{\beta} R_{r_1+1},\dots,(x+c)^{\beta} R_{r}\right]
\end{equation}
is a polynomial. Moreover, if  $\beta \neq \deg R_i - \deg R_{r_1+j}$ for  $1\leq i \leq r_1$ and  $1\leq j \leq r_2$, then $\Omega$ is a polynomial of degree $\sum\limits_{i=1}^{r} \deg R_{i} - \frac{r_1(r_1-1)}{2} - \frac{r_2(r_2-1)}{2}$ with leading coefficient
\begin{equation}\label{eq:OmegaLeaCoe}
	\prod_{1\leq i < j \leq r_1}\left(\deg R_j - \deg R_i \right) \prod_{1\leq i < j \leq r_2}\left(\deg R_{r_1+j} - \deg R_{r_1+i} \right) \prod_{i=1}^{r_1}\prod_{j=1}^{r_2}\left(\beta-\deg R_i + \deg R_{r_1+j}\right).
\end{equation}
\end{proposition}

If we specify Proposition \ref{prop:degwr} to Jacobi polynomials, we get the result of Lemma \ref{lem:deg}. Then, the assumptions in the proposition coincide with the conditions \eqref{Condition1GJP}. 

The result of Lemma \ref{lem:deg} was already proven by Dur\'an. He also approached via general polynomials and then specified to Jacobi polynomials. Hence our Proposition \ref{prop:degwr} can be compared to his general statement \cite[Lemma 3.3]{Duran-a}. The main difference between our approach and Dur\'an's, is that he used a limiting procedure whereas our approach is more direct.

Proposition \ref{prop:degwr} tells us that there is a degree reduction if the conditions for $\beta$ are not satisfied, then the coefficient \eqref{eq:OmegaLeaCoe} would vanish. Moreover, the leading coefficient \eqref{eq:OmegaLeaCoe} is in fact a specific Vandermonde determinant as shown in the proof of Lemma \ref{lem:WronskianMonomials} later on. The proof is given in Section \ref{sec:ProofDegWr}.

Next, we state another result which we use in the proof of Lemma \ref{lem:XJPlincomJac}. The proof itself is in Section \ref{sec:ProofDegQ}.

\begin{proposition}\label{prop:degQ}
Let $r_1,r_2$ be non-negative integers and set $r=r_1+r_2$. Let $R_1,\dots,R_r,R$ be non-zero monic polynomials such that $\deg R_i \neq \deg R_{j}$ for $i\neq j$ and $1\leq i,j \leq r_1$ or $r_1+1 \leq i,j \leq r$. Take $c,\beta\in\mathbb{R}$ and fix $k\in\{0,\dots,r\}$. Let $C_k$ denote the $r\times r$-matrix obtained by deleting the $(k+1)$-th row and $(r+1)$-th column of matrix linked to the Wronskian  $\Wr\left[R_1,\dots,R_{r_1},(x+c)^{\beta} R_{r_1+1},\dots,(x+c)^{\beta} R_{r}, R\right]$. Then the function
\begin{equation*}
	Q_k 
		:= (x+c)^{(r_1+1-\beta)r_2}  \cdot \det(C_k)
\end{equation*}
is a polynomial. Moreover, if $\beta \neq \deg R_i - \deg R_{r_1+j}$ for  $1\leq i \leq r_1$ and  $1\leq j \leq r_2$, then $Q_k$ is a polynomial with degree at most $\sum\limits_{i=1}^{r} \deg R_{i} - \frac{r_1(r_1-1)}{2} - \frac{r_2(r_2-1)}{2} -r_1+k$.
\end{proposition}

When $k=r$, then $Q_r$ is closely related to $\Omega$ defined in \eqref{eq:OmegaDef}. In that case, $Q_r$ consists of the similar Wronskian as $\Omega$ but differs in the prefactor, i.e., $Q_r=(x+c)^{r_2}\cdot \Omega$. Nevertheless, both degree statements of $Q_r$ and $\Omega$ agree with each other, although Proposition \ref{prop:degwr} is stronger. A similar argument can be made for $k=0$, then $Q_0$ consists of the Wronskian with elements which are the derivatives of the elements in the Wronskian of $\Omega$. 

\subsection{Proof of Proposition \ref{prop:degwr}}\label{sec:ProofDegWr}
We start with a lemma considering the Wronskian of monomials. Next, we prove Proposition \ref{prop:degwr} using this lemma various times. 

\begin{lemma}\label{lem:WronskianMonomials}
Take two sets of non-negative integers $k_1,\dots,k_{r_1}$ and $l_1,\dots,l_{r_2}$. Set $r=r_1+r_2$. Then
\begin{equation}\label{eq:WronskianMonomials}
	\Wr[x^{k_1},\dots,x^{k_{r_1}},x^{l_1+\beta},\dots,x^{l_{r_2}+\beta}]
		= e \cdot x^{N} 
\end{equation}
where
\begin{align}
	N
		&=\sum\limits_{i=1}^{r_1}k_i + \sum\limits_{j=1}^{r_2}l_j + \beta r_2- \frac{r(r-1)}{2},
		\label{eq:WronskianMonomialsDegree}\\
	e
		&= \prod_{1\leq i < j \leq r_1}(k_j-k_i) \prod_{1\leq i < j \leq r_2}(l_j-l_i) \prod_{i=1}^{r_1}\prod_{j=1}^{r_2}(\beta+l_j-k_i).
		\label{eq:WronskianMonomialsConstant}
\end{align}
\end{lemma}
\begin{proof}
Consider the Wronskian \eqref{eq:WronskianMonomials}. The determinant expression via the sum of all permutations gives that each element consists of the monomial $x^N$ where $N$ is given by \eqref{eq:WronskianMonomialsDegree}. Hence, we directly find identity \eqref{eq:WronskianMonomials}. To derive $e$, evaluate the Wronskian in 1 such that we obtain 
\begin{equation*}
	e=
	\begin{vmatrix}
		1 & \dots & 1 & 1 & \dots & 1 \\
		k_1 & \dots & k_{r_1} & l_1+\beta & \dots & l_{r_2}+\beta \\
		k_1(k_1-1) & \dots & k_{r_1}(k_{r_1}-1) & (l_1+\beta)(l_1+\beta-1) & \dots & (l_{r_2}+\beta)(l_{r_2}+\beta-1) \\
		\vdots & \ddots & \vdots & \vdots & \ddots & \vdots \\
		(k_1)_{r-1} & \dots & (k_{r_1})_{r-1} & (l_1+\beta)_{r-1} & \dots & (l_{r_2}+\beta)_{r-1}
	\end{vmatrix}
\end{equation*}
where we use the Pochhammer symbol. After elementary row operations, we obtain that $e$ is the Vandermonde determinant of the elements $k_1,\dots,k_{r_1},l_1+\beta,\dots,l_{r_2}+\beta$. This determinant is well-known, it is given by the product of all differences of each two elements. If we write out this value, we easily obtain \eqref{eq:WronskianMonomialsConstant}.
\end{proof}

Now we are able to prove Proposition \ref{prop:degwr}. Consider the set-up as in the proposition. Let $r_1,r_2$ be non-negative integers and define $r=r_1+r_2$. Fix $R_1,\dots,R_r$ as non-zero monic polynomials such that $\deg R_i \neq \deg R_{j}$ for $i\neq j$ and $1\leq i,j \leq r_1$ or $r_1+1 \leq i,j \leq r$. Fix $c,\beta\in\mathbb{R}$. 

\begin{proof}[Proof of Proposition \ref{prop:degwr}]
Without loss of generality, we take $c=0$. When $c\neq0$, we would shift the variable $x\to x-c$ which does not influence the degree nor the leading coefficient of the polynomials $R_1,\dots,R_r$.

Recall the definition of $\Omega$ (where we take $c=0$), 
\begin{equation}\label{eq:OmegaDef2}
	\Omega 
		= x^{(r_1-\beta)r_2} \cdot \Wr\left[R_1,\dots,R_{r_1},x^{\beta} R_{r_1+1},\dots,x^{\beta} R_{r}\right].
\end{equation}
We first simplify our notation, set
\begin{align*}
	\deg(R_j) &= n_j && 1\leq j \leq r_1, \\
	\deg(R_{r_1+j})&= m_j && 1\leq j \leq r_2,
\end{align*}
and as we are dealing with monic polynomials, we may write
\begin{align*}
	& R_j(x) 
		=x^{n_j} + \sum_{k=0}^{n_j-1} c_{j,k} x^k && 1\leq j \leq r_1, \\
	& R_{r_1+j}(x) 
		=x^{m_j} + \sum_{l=0}^{m_j-1} d_{j,l} x^l && 1\leq j \leq r_2.
\end{align*}

Next, we use the multilinearity of the determinant to expand the Wronskian in \eqref{eq:OmegaDef2}, we obtain
\begin{multline}\label{eq:DegreeProof1}
	\Wr[x^{n_1},\dots,x^{n_{r_1}},x^{m_1+\beta},\dots,x^{m_{r_2}+\beta}] + \\
		 \sum_{k_1=0}^{n_1-1} \dots \sum_{k_{r_1}=0}^{n_{r_1}-1} \sum_{l_1=0}^{m_1-1} \dots  \sum_{l_{r_2}=0}^{m_{r_2}-1} c_{1,k_1}\dots  c_{r_1,k_{r_1}} d_{1,l_1} \dots d_{r_2,l_{r_2}}
		 	\Wr[x^{k_1},\dots,x^{k_{r_1}},x^{l_1+\beta},\dots,x^{l_{r_2}+\beta}]
\end{multline}
For each Wronskian in \eqref{eq:DegreeProof1}, there are two possibilities. Either Lemma \ref{lem:WronskianMonomials} tells us that it equals a monomial multiplied by a constant, or it vanishes if two entries in the Wronskian are the same. We start with the first Wronskian in \eqref{eq:DegreeProof1} and next we consider the Wronskian in the multiple sum.

If we apply Lemma \ref{lem:WronskianMonomials} on the first Wronskian in \eqref{eq:DegreeProof1}, we get
\begin{equation*}
	\Wr[x^{n_1},\dots,x^{n_{r_1}},x^{m_1+\beta},\dots,x^{m_{r_2}+\beta}]
		= e_N \cdot x^N
\end{equation*}
with
\begin{align}
	N
		&= \sum_{i=1}^{r_1}n_i + \sum_{j=1}^{r_2}m_j + \beta r_2- \frac{r(r-1)}{2}, 
		\label{eq:DegreeProofN}\\
	e_N
		&= \prod_{1\leq i < j \leq r_1}\left(n_j - n_i \right) \prod_{1\leq i < j \leq r_2}\left(m_j - m_i \right) \prod_{i=1}^{r_1}\prod_{j=1}^{r_2}\left(\beta-n_i + m_j\right).
		\label{eq:DegreeProofeN}
\end{align}

Next, consider a Wronskian in the multiple sum of \eqref{eq:DegreeProof1}. According to Lemma \ref{lem:WronskianMonomials}, this Wronskian vanishes when two entries are the same as then the coefficient \eqref{eq:WronskianMonomialsConstant} equals zero. Therefore, we may assume (possibly after rearranging terms), $k_1> k_2 > \dots > k_{r_1} \geq 0$ and $l_1> l_2 > \dots > l_{r_2} \geq 0$. We get  
\begin{equation*}
 	\Wr[x^{k_1},\dots,x^{k_{r_1}},x^{l_1+\beta},\dots,x^{l_{r_2}+\beta}]
		= e_{k_1,\dots,l_{r_2}} x^{N_{k_1,\dots,l_{r_2}}}
\end{equation*}
for some constant $e_{k_1,\dots,l_{r_2}}\in\mathbb{R}$ and where
\begin{equation}\label{eq:DegreeProofN2}
	N_{k_1,\dots,l_{r_2}}
		= \sum_{i=1}^{r_1} k_i + \sum_{j=1}^{r_2} l_j +\beta r_2 - \frac{r(r-1)}{2}.
\end{equation}
The constant $e_{k_1,\dots,l_{r_2}}$ still can be equal to zero for specific values of the parameter $\beta$.

Now we argue that $\Omega$ is a polynomial. Therefore, we must take into account the monomial prefactor in \eqref{eq:OmegaDef2}. As all the terms in \eqref{eq:DegreeProof1} are monomials, we just get a shift in the exponents \eqref{eq:DegreeProofN} and \eqref{eq:DegreeProofN2} if we multiply each term by this prefactor. A short calculation gives 
\begin{align*}
	\tilde{N}
		&:= N + (r_1-\beta)r_2 = \sum_{i=1}^{r} \deg R_{i} - \frac{r_1(r_1-1)}{2} - \frac{r_2(r_2-1)}{2} \geq 0, \\
	\tilde{N}_{k_1,\dots,l_{r_2}}
		&:= N_{k_1,\dots,l_{r_2}} + (r_1-\beta)r_2 \geq 0.	
\end{align*}
The obtained values $\tilde{N}$ and $\tilde{N}_{k_1,\dots,l_{r_2}}$ trivially are integers and one directly sees that $\tilde{N}\geq 0$. The result that $\tilde{N}_{k_1,\dots,l_{r_2}}\geq0$ for all choices of $k_1,\dots,l_{r_2}$ follows from the fact that the lowest possible value is obtained when $k_j=r_1-j$ for all $j=1,\dots,r_1$ and $l_j=r_2-j$ for all $j=1,\dots,r_2$. In this case, $\tilde{N}_{k_1,\dots,l_{r_2}}=0$. We also find 
\begin{equation*}
	\tilde{N} > \tilde{N}_{k_1,\dots,l_{r_2}}
\end{equation*}
as $k_i < n_i$ for all $i=1,\dots,r_1$ and $l_j<m_j$ for all $j=1,\dots,r_2$. Therefore, $\Omega$ is a polynomial as all exponents of the monomials are non-negative integers, and the leading term is given by $x^{\tilde{N}}$ as long as the corresponding coefficient $e_N$, see \eqref{eq:DegreeProofeN}, does not vanish. Hence, we impose conditions for the parameter $\beta$ such that $e_N\neq 0$. Under these conditions, $\Omega$ has degree $\tilde{N}$ with leading coefficient $e_N$ which ends the proof.
\end{proof}

\subsection{Proof of Proposition \ref{prop:degQ}}\label{sec:ProofDegQ}

\begin{lemma}\label{lem:WronskianMonomials2}
Take two sets of non-negative integers $k_1,\dots,k_{r_1}$ and $l_1,\dots,l_{r_2}$. Set $r =r_1+r_2$. Fix $k\in\{0,1\dots,r \}$. Let $C_k$ denote the $r\times r$-matrix obtained by deleting the $(k+1)$-th row and $(r+1)$-th column of matrix linked to the Wronskian  $\Wr[x^{k_1},\dots,x^{k_{r_1}},x^{l_1+\beta},\dots,x^{l_{r_2}+\beta},1]$. Then 
\begin{equation*}
	\det(C_k)
		= e \cdot x^{N},
\end{equation*}
where $N=\sum\limits_{i=1}^{r_1}k_i + \sum\limits_{j=1}^{r_2}l_j + \beta r_2- \frac{r(r+1)}{2}+k$ and for some constant $e\in\mathbb{R}$.
\end{lemma}
\begin{proof}
Consider the Wronskian. The determinant expression via the sum of all permutations gives that each element consists of the monomial $x^N$ and a constant. Hence, we get that the Wronskian equals $e \cdot x^N$ for some $e\in\mathbb{R}$. The constant $e$ can be determined explicitly, but we omit further specification. 
\end{proof}

Now we prove Proposition \ref{prop:degQ}. As the ideas are completely similar to the proof of Proposition \ref{prop:degwr}, we only highlight the differences. As we are only interested in an upper bound for the degree, we do not investigate the corresponding leading coefficient.

\begin{proof}[Proof of Proposition \ref{prop:degQ}]
Again, we may assume that $c=0$. Fix $k\in\{0,1,\dots,r\}$. We use the same notation (as in the proof of Proposition \ref{prop:degwr}) and we expand the polynomials $R_j$ for $j=1,\dots,r$ in the same way.

Again we use the multilinearity of the determinant to decompose the Wronskian, we get a similar expansion as in \eqref{eq:DegreeProof1}. Now we delete the $(k+1)$-th row and $(r+1)$-th column in each Wronskian. Then we can apply Lemma \ref{lem:WronskianMonomials2} to get that each Wronskian is a monomial multiplied by a constant. The degrees are
\begin{align*}
	N
		&= \sum_{i=1}^{r_1}n_i + \sum_{j=1}^{r_2}m_j + \beta r_2- \frac{r(r-1)}{2} +k, \\
	N_{k_1,\dots,l_{r_2}}
		&= \sum_{i=1}^{r_1} k_i + \sum_{j=1}^{r_2} l_j +\beta r_2 - \frac{r(r-1)}{2} +k.	
\end{align*}

Next, we take into account the prefactor and obtain that the degrees of the monomials shift to
\begin{align*}
	\tilde{N}
		&:= N + (r_1+1-\beta)r_2 
		= \sum\limits_{i=1}^{r} \deg R_{i} - \frac{r_1(r_1-1)}{2} - \frac{r_2(r_2-2)}{2} -r_1+k \geq 0,\\
	\tilde{N}_{k_1,\dots,l_{r_2}}
		&:= N_{k_1,\dots,l_{r_2}} + (r_1+1-\beta)r_2 \geq 0.	
\end{align*}
Again the values are non-negative integers and $\tilde{N}>\tilde{N}_{k_1,\dots,l_{r_2}}\geq 0$, hence we are dealing with a polynomial of degree at most $\tilde{N}$. As the leading coefficient could vanish, we only obtain an upper bound for the degree. This ends the proof.
\end{proof}

\section{Construction of the generalized Jacobi polynomial}\label{sec:ConstructionGJP}
In this section we reformulate and prove the result of Takemura in \cite{Takemura-Maya}. He showed that a Wronskian, containing all four types of eigenfunctions of Table \ref{tab:1}, is essentially equal to a Wronskian consisting of only two types of eigenfunctions. To be precise, let
\begin{align}
	f_j(x) = &P^{(\alpha,\beta)}_{n_j}(x), &&j=1,\dots,r_1,
	\label{app:fj1}\\
	f_{r_1+j}(x) = &(1+x)^{-\beta}P^{(\alpha,-\beta)}_{m_j}(x), &&j=1,\dots,r_2,
	\label{app:fj2}\\
	f_{r_1+r_2+j}(x) = &(1-x)^{-\alpha} P^{(-\alpha,\beta)}_{m'_j}(x), &&j=1,\dots,r_3,
	\label{app:fj3}\\
	f_{r_1+r_2+r_3+j}(x) = &(1-x)^{-\alpha}(1+x)^{-\beta} P^{(-\alpha,-\beta)}_{n'_j}(x) &&j=1,\dots,r_4,
	\label{app:fj4}
\end{align}
with $r_1 + r_2 + r_3 + r_4 = r$, $n_1 > n_2 > \cdots > n_{r_1} \geq 0$,
$m_1 > m_2 > \cdots > m_{r_2} \geq 0$, $m_1' > m_2' > \cdots > m_{r_3}' \geq 0$,
and $n_1' > n_2' > \cdots > n_{r_1}' \geq 0$. We define 
\begin{equation}\label{eq:Omega}
	\Psi := (1+x)^{(\beta+r_1+r_3)(r_2+r_4)}(1-x)^{(\alpha+r_1+r_2)(r_3+r_4)} 
		\Wr \left[ f_1, f_2, \ldots, f_r \right]. 
\end{equation}
The result is that whenever $\Psi$ does not vanish, there are two partitions $\lambda,\mu$ and parameters $\alpha',\beta'$ such that $\Psi = c \cdot \Omega^{(\alpha',\beta')}_{\lambda,\mu}$ for some constant $c$ and where $\Omega^{(\alpha',\beta')}_{\lambda,\mu}$ is the generalized Jacobi polynomial defined in \eqref{eq:OmegaLaMu}. We therefore conclude in particular that $\Psi$ is a polynomial too. 

The procedure to obtain this result can be visualized by Maya diagrams. We shortly recall a few notions of Maya diagrams in Section \ref{sec:MayaDiagrams}. Next, we state the result properly in Section \ref{sec:StatementResult} and prove it in Section \ref{sec:ProofShiftMaya}. The proof is a straightforward generalization of the argument given in \cite{GomezUllate_Grandati_Milson-a} for the case of Exceptional Hermite polynomials which consist of 1 Maya diagram. Finally we consider Lemma \ref{lem:ConPar} as a special case of our theorem, see Section \ref{sec:SpecialCase}. 

\begin{remark}
The result we prove in this section, see Theorem \ref{thm:ShiftMaya} below, also holds true in the Hermite case \cite{GomezUllate_Grandati_Milson-a} and in the Laguerre case \cite[Section 4]{Bonneux_Kuijlaars}. Recently, the authors which prove the result for the Hermite case extend their work to the Laguerre and Jacobi case as well, see \cite{GomezUllate_Grandati_Milson-L+J}. Therefore our result coincides with identity (136) in that paper.
\end{remark}

\subsection{Maya diagrams}\label{sec:MayaDiagrams}
We recall the necessary concepts of Maya diagrams to present our result. 

\paragraph{Definition Maya diagram.}
A Maya diagram $M$ is a subset of the integers that contains a finite number of positive integers and excludes a finite number of negative integers. We visualize it as an infinite row of boxes which are empty or filled. 
We order these boxes by corresponding them to the set of integers and therefore we define an origin. To the right of the origin, there are only finitely many filled boxes. Each of these filled boxes corresponds to a non-negative integer $a\geq0$. All filled boxes to the right of the origin are labeled by a finite
decreasing  sequence 
\begin{equation*}
	a_1 > a_2 > \cdots > a_{r_1} \geq 0
\end{equation*}
where $r_1$ is the number of filled boxes to the right of the origin. 

To the left of the origin, there are only finitely many empty boxes. Each empty box corresponds to a negative integer $k<0$. We link this negative integer to a non-negative integer $a' = -k-1 \geq 0$. We obtain
a second finite decreasing sequence 
\begin{equation*}
	a_1' >  a_2' >  \cdots > a_{r_4}' \geq 0
\end{equation*}
where $r_4$ is the number of those boxes. These numbers label the positions of the empty boxes to the left of the origin. The Maya diagram is encoded by these two sequences
\begin{equation} \label{eq:Mayacoding}
	M : \quad \left( a_1', a_2', \ldots, a_{r_4}' \mid a_1, a_2, \ldots, a_{r_1} \right).
\end{equation}

\paragraph{Equivalent Maya diagrams.}
We say that a Maya diagram $\widetilde{M}$ is equivalent to $M$ if $\widetilde{M}$ is  
obtained from $M$ by moving the position of the origin, i.e., $\widetilde{M}=M-t$ where $t$ denotes the shift of the origin. Hence, the sequence of filled and empty boxes remain unchanged. We specify two canonical choices.

\begin{itemize}
	\item[1.]
	We can put the origin such that all boxes to the left are filled, while the first box to the right is empty. Since there are no empty boxes on the left, we only have 1 decreasing sequence, say
	\begin{equation}\label{eq:MayaCanonicalFormA}
		\widetilde{M} : \quad \left(\emptyset \mid n_1, n_2, \ldots, n_r \right)
	\end{equation}
	with $n_r \geq 1$. This strictly decreasing sequence $(n_i)_{i=1}^{r}$ is associated with a partition $\lambda$ as before, we denote it by  
	\begin{equation} \label{eq:lambdaM} 
		\lambda = \lambda(M). 
	\end{equation}
	Moreover, we use 
	\begin{equation} \label{eq:tM}
		t = t(M) 
	\end{equation}
	to denote the shift $M \to \widetilde{M} = M+t$ that we apply to $M$ to shift it into this canonical form. If $M$ is encoded by \eqref{eq:Mayacoding} then
	\begin{equation*}
	t(M) 
		=\begin{cases}
			a_1' + 1	& r_4>0, \\
			-\max\{i \mid a_i=i-1\} & r_4 = 0 \text{ and } a_1=0, \\
			0	&	r_4 = 0 \text{ and } a_1\neq 0.
		\end{cases}
	\end{equation*}
	
	\item[2.]
	The origin can also be placed so that all the boxes to the right are empty,	and the first box to the left is filled. This has an encoding 
	\begin{equation}\label{eq:MayaCanonicalFormB}
		\widetilde{M} : \quad (n_1', n_2', \ldots, n_s' \mid \emptyset)
	\end{equation}
	with $n_s' \geq 1$. Then 
	\begin{equation*}
		\lambda'_j = n_j' - s + j, \qquad j =1, \ldots,s,
	\end{equation*}
	is the partition that is conjugate to $\lambda(M)$.
\end{itemize}

\subsection{Formulation of the result with Maya diagrams}\label{sec:StatementResult}
Consider the polynomial $\Psi$ defined in \eqref{eq:Omega}. We redefine this polynomial using Maya diagrams. We set
\begin{equation} \label{eq:M1M2}
	\begin{aligned}
		M_1: & \quad \left( n'_1,\dots,n'_{r_4} \mid n_1,\dots,n_{r_1} \right), \\
		M_2: & \quad \left( m'_1,\dots,m'_{r_3} \mid m_1,\dots,m_{r_2} \right),
	\end{aligned}
\end{equation} 
as two Maya diagrams built out of the degrees of the Jacobi polynomials appearing in \eqref{app:fj1}-\eqref{app:fj4}. We rewrite $\Psi$ as
\begin{equation}\label{eq:OmegaM1M2}
	\Omega^{(\alpha,\beta)}_{M_1,M_2}
		= (1+x)^{(\beta+r_1+r_3)(r_2+r_4)}(1-x)^{(\alpha+r_1+r_2)(r_3+r_4)} 
		 \cdot \Wr \left[ f_1, f_2, \ldots, f_r \right]. 
\end{equation}
We set $\widehat{\Omega}^{(\alpha,\beta)}_{M_1,M_2}$ as its monic variant. 

Without any conditions for the parameters, the function $\Omega^{(\alpha,\beta)}_{M_1,M_2}$ can vanish, we therefore impose the following conditions.

\begin{itemize}	
	\item[1.] \textbf{No degree reduction} \\
	We want to make sure the Jacobi polynomials do not have any degree reduction, the following conditions are necessary and sufficient,
	\begin{equation}\label{MayaCondition2GJP}
	\begin{aligned}
		\alpha+\beta+n_i &\notin \{-1,-2,\dots,-n_i\},		&&\qquad i=1,\dots,r_1, \\
		\alpha-\beta+m_i &\notin \{-1,-2,\dots,-m_i\},		&&\qquad i=1,\dots,r_2, \\ 
		-\alpha+\beta+m_i' &\notin \{-1,-2,\dots,-m_i'\},	&&\qquad i=1,\dots,r_3, \\
		-\alpha-\beta+n_i' &\notin \{-1,-2,\dots,-n_i'\},	&&\qquad i=1,\dots,r_4. 
	\end{aligned}
	\end{equation}

	\item[2.] \textbf{Independent eigenfunctions} \\
	Under conditions \eqref{MayaCondition2GJP}, the Wronskian in \eqref{eq:OmegaM1M2} consists of independent eigenfunctions if and only if
	\begin{equation}\label{MayaCondition1GJP}
	\begin{aligned}
		&\beta \neq m_j-n_i, 	&&\qquad i=1,\dots,r_1 \text{ and } j=1,\dots,r_2, \\
		&\beta \neq n_j'-m_i', 	&&\qquad i=1,\dots,r_4 \text{ and } j=1,\dots,r_3, \\
		&\alpha \neq m_j'-n_i, 	&&\qquad i=1,\dots,r_1 \text{ and } j=1,\dots,r_3, \\
		&\alpha \neq n_j'-m_i, 	&&\qquad i=1,\dots,r_4 \text{ and } j=1,\dots,r_2.
	\end{aligned}	
	\end{equation}
	This follows directly as then all eigenpolynomials have a different degree in the broad sense, $\deg\left( (1-x)^{\alpha} (1+x)^{\beta} P^{(\alpha,\beta)}_{m}\right) = \deg\left( P^{(\alpha,\beta)}_{m}\right)+\alpha+\beta$.
\end{itemize}

Now we can state our result in detail. Note that the theorem gives that $\Omega^{(\alpha,\beta)}_{M_1,M_2}$ is a polynomial.

\begin{theorem}\label{thm:ShiftMaya}
Let $f_1, \ldots, f_r$ be as in  \eqref{app:fj1}-\eqref{app:fj4}, and let $M_1,M_2$ be two Maya diagrams built of the degrees as in \eqref{eq:M1M2}. Let $\lambda = \lambda(M_1)$ and $\mu = \lambda(M_2)$ be the two partitions that are associated with $M_1$ and $M_2$, see \eqref{eq:lambdaM}, and let $t_1 = t(M_1)$, $t_2 = t(M_2)$, see \eqref{eq:tM}. Take $\alpha,\beta\in\mathbb{R}$ such that the conditions \eqref{MayaCondition2GJP} and \eqref{MayaCondition1GJP} are satisfied.
Then 
\begin{equation*}
	\widehat{\Omega}^{(\alpha,\beta)}_{M_1,M_2}
		= \widehat{\Omega}^{(\alpha-t_1-t_2,\beta-t_1+t_2)}_{\lambda, \mu}.
\end{equation*}
\end{theorem}

As the result is formulated in terms of monic polynomials, it is also true for the non-monic polynomials although we then need a constant to establish the identity. Contrary to the Laguerre case, see \cite{Bonneux_Kuijlaars}, we do not formulate an explicit formula for this constant.

\subsection{Proof of Theorem \ref{thm:ShiftMaya}}\label{sec:ProofShiftMaya}
The main idea for this proof is that we shift both Maya diagrams $M_1$ and $M_2$ to its canonical form A, see \eqref{eq:MayaCanonicalFormA}. This transformation leads to a shift in both parameters which we explain in the following lemma. We formulate how the polynomial \eqref{eq:OmegaM1M2} changes when the origin in a Maya diagram shifts by 1 step and then, we use these results repeatedly to obtain the result of Theorem \ref{thm:ShiftMaya}. 

\begin{lemma}\label{lem:MayaShiftOneStep}
Let $M_1$ and $M_2$ be given by \eqref{eq:M1M2} and take $\alpha,\beta\in\mathbb{R}$ such that the conditions \eqref{MayaCondition2GJP} and \eqref{MayaCondition1GJP} are satisfied. 
\begin{enumerate}
	\item[\rm (a)] If $n_{r_1}=0$, then 
	\begin{equation} \label{eq:reduction1}
		\widehat{\Omega}^{(\alpha,\beta)}_{M_1,M_2}
		= \widehat{\Omega}^{(\alpha+1,\beta+1)}_{M_1-1,M_2}.
	\end{equation}
	\item[\rm (b)] If $n'_{r_4}=0$, then
	\begin{equation} \label{eq:reduction4}
		\widehat{\Omega}^{(\alpha,\beta)}_{M_1,M_2}
		= \widehat{\Omega}^{(\alpha-1,\beta-1)}_{M_1+1,M_2}.
	\end{equation}
	\item[\rm (c)]  If $m_{r_2} = 0$, then
	\begin{equation} \label{eq:reduction2}
		\widehat{\Omega}^{(\alpha,\beta)}_{M_1,M_2}
		= \widehat{\Omega}^{(\alpha+1,\beta-1)}_{M_1,M_2-1}.
	\end{equation}
	\item[\rm (d)]  If $m'_{r_3} = 0$, then
	\begin{equation} \label{eq:reduction3}
		\widehat{\Omega}^{(\alpha,\beta)}_{M_1,M_2}
		= \widehat{\Omega}^{(\alpha-1,\beta+1)}_{M_1,M_2+1}.
	\end{equation}
\end{enumerate}
\end{lemma}
\begin{proof}
The four items are proven in a similar fashion and therefore we only give the proof of part (a). See also the proof of Lemma 9 in \cite{Bonneux_Kuijlaars} for the Laguerre case.

Under conditions \eqref{MayaCondition2GJP} and \eqref{MayaCondition1GJP}, the polynomial $\widehat{\Omega}^{(\alpha,\beta)}_{M_1,M_2}$ does not vanish. Now, if $n_{r_1}=0$, then $M_1-1$ is encoded by
\begin{equation*}
	M_1-1: \quad \left( n_1'+1, n_2'+1, \ldots, n_{r_4}'+1 \mid n_1-1, n_2-1, \ldots, n_{r_1-1}-1 \right).
\end{equation*}
Therefore, the conditions \eqref{MayaCondition1GJP} in terms of $\widehat{\Omega}^{(\alpha+1,\beta+1)}_{M_1-1,M_2}$ are satisfied too. Moreover, there is also no degree reduction for the Jacobi polynomials, i.e., after an elementary calculation and using $n_{r_1}=0$ we find that \eqref{MayaCondition2GJP} in terms of $\widehat{\Omega}^{(\alpha+1,\beta+1)}_{M_1-1,M_2}$ is valid too. Hence, the polynomial $\widehat{\Omega}^{(\alpha+1,\beta+1)}_{M_1-1,M_2}$ does not vanish either.

To establish the identity \eqref{eq:reduction1}, one uses
\begin{equation*}
	\Wr[f_1,\dots,f_{r_1-1}, 1, f_{r_1+1} ,\dots, f_r]	
		= (-1)^{r_1+1} \Wr[f'_1,\dots,f'_{r_1-1}, f'_{r_1+1} ,\dots, f'_r]
\end{equation*}
where the (general) derivatives for \eqref{app:fj1}-\eqref{app:fj4} are given by
\begin{align}
	&\frac{d}{dx} \left( P^{(\alpha,\beta)}_{n}(x) \right) 
		= \frac{n+\alpha+\beta+1}{2} P^{(\alpha+1,\beta+1)}_{n-1}(x), 
	\label{eq:diff1} \\
	&\frac{d}{dx}\left((1+x)^{-\beta}P^{(\alpha,-\beta)}_{n}(x)\right)
		= (n-\beta) (1+x)^{-\beta-1}P^{(\alpha+1,-\beta-1)}_{n}(x), 
	\nonumber \\
	&\frac{d}{dx}\left((1-x)^{-\alpha}P^{(-\alpha,\beta)}_{n}(x)\right)
		= -(n-\alpha) (1-x)^{-\alpha-1}P^{(-\alpha-1,\beta+1)}_{n}(x), 
	\label{eq:diff3} \\
	&\frac{d}{dx} \left( (1-x)^{-\alpha}(1+x)^{-\beta}P^{(-\alpha,-\beta)}_{n}(x)\right) 
		= -2(n+1) (1-x)^{-\alpha-1}(1+x)^{-\beta-1}P^{(-\alpha-1,-\beta-1)}_{n+1}(x). 
	\nonumber
\end{align}
These identities follow directly from the Rodrigues' formula \eqref{eq:Rodrigues}. Finally, one can take out common factors from each column in the Wronskian such that the prefactors coincide and identity \eqref{eq:reduction1} is obtained. Moreover, we do not need a constant to establish the identity as both polynomials are monic.
\end{proof}	

Lemma \ref{lem:MayaShiftOneStep} gives sufficient information to express $\widehat{\Omega}^{(\alpha,\beta)}_{M_1,M_2}$ in terms of $\widehat{\Omega}^{(\alpha',\beta')}_{\widetilde{M}_1,\widetilde{M}_2}$ where $\widetilde{M}_1,\widetilde{M}_2$ are equivalent Maya diagrams of $M_1$ and $M_2$ respectively. If $\widetilde{M}_1=M_1+t_1$, then one can use $|t_1|$ times identity \eqref{eq:reduction1} or \eqref{eq:reduction4}, and if $\widetilde{M}_2=M_1+t_2$, then one can use $|t_2|$ times identity \eqref{eq:reduction2} or \eqref{eq:reduction3}.

\begin{proof}[Proof of Theorem \ref{thm:ShiftMaya}]
Let $t_1=t(M_1)$. We want to shift the origin of the Maya diagram $M_1$ by $|t_1|$ places such that we end up with $M_1+t_1$. We do this transformation in $|t_1|$ steps, in each step we shift the origin by one place. Assume that $t_1\geq0$. When $t_1$ is negative, a similar reasoning holds true (it would make it even easier).

Consider Maya diagram $M_1$ as in \eqref{eq:M1M2}, there are two scenarios.
%
\begin{enumerate}
	\item[(1)]
	The box to the left of the origin is empty. 
	This holds true if and only if $n'_{r_4}=0$.
	
	\item[(2)]
	The box to the left of the origin is filled.
	This holds true if and only if $n'_{r_4}>0$. Stated otherwise, in terms of the Maya diagram $M_1+1$ we have
	\begin{equation*}
		M_1+1: \quad \left( n_1'-1, n_2'-1, \ldots, n_{r_4}'-1 \mid n_1+1, n_2+1, \ldots, n_{r_1}+1,0 \right).
	\end{equation*}
\end{enumerate}

In the first scenario, we apply \eqref{eq:reduction4} for the Maya diagrams $M_1$ and $M_2$. We obtain the identity
\begin{equation*}
	\widehat{\Omega}^{(\alpha,\beta)}_{M_1,M_2}
		= \widehat{\Omega}^{(\alpha-1,\beta-1)}_{M_1+1,M_2}
\end{equation*}
For the second scenario, we apply \eqref{eq:reduction1} for the Maya diagrams $M_1+1$ and $M_2$. We get
\begin{equation*}
	\widehat{\Omega}^{(\alpha-1,\beta-1)}_{M_1+1,M_2}
		= \widehat{\Omega}^{(\alpha,\beta)}_{M_1,M_2}
\end{equation*}
Both scenarios give the same result, the parameters always shift in the same direction (they both decrease by 1). 

Subsequently, we apply the same argument on the Maya diagram $M_1+1$ as either its box to the left of the origin is filled or empty. Next, we do the same for $M_1+2$. Repeating this procedure $t_1$ times gives
\begin{equation}\label{eq:shifting1}
	\widehat{\Omega}^{(\alpha,\beta)}_{M_1,M_2}
		= \widehat{\Omega}^{(\alpha-1,\beta-1)}_{M_1+1,M_2}
		= \cdots
		= \widehat{\Omega}^{(\alpha-t_1,\beta-t_1)}_{M_1+t_1,M_2}.
\end{equation}
Next, let $t_2=t(M_2)$ and assume $t_2\geq0$. A similar argument as before, but now using \eqref{eq:reduction2} and \eqref{eq:reduction3}, gives
\begin{equation}\label{eq:shifting2}
	\widehat{\Omega}^{(\alpha-t_1,\beta-t_1)}_{M_1+t_1,M_2}
		= \widehat{\Omega}^{(\alpha-t_1-1,\beta-t_1+1)}_{M_1+t_1,M_2+1}
		= \cdots
		= \widehat{\Omega}^{(\alpha-t_1-t_2,\beta-t_1+t_2)}_{M_1+t_1,M_2+t_2}
\end{equation}
Hence if we combine \eqref{eq:shifting1} and \eqref{eq:shifting2}, we end up with
\begin{equation*}
	\widehat{\Omega}^{(\alpha,\beta)}_{M_1,M_2}
		=\widehat{\Omega}^{(\alpha-t_1-t_2,\beta-t_1+t_2)}_{M_1+t_1,M_2+t_2}
\end{equation*}
which is the result we wanted to prove.
\end{proof}

\begin{remark} 
The proof of Theorem \ref{thm:ShiftMaya} consists of shifting each Maya diagram to its canonical form \eqref{eq:MayaCanonicalFormA}. However, one could also shift 1 (or both) Maya diagram(s) to the other form \eqref{eq:MayaCanonicalFormB}. Hence, in general one can reduce the polynomial \eqref{eq:OmegaM1M2} to a well-chosen prefactor and Wronskian consisting of functions of the form \eqref{app:fj1} or \eqref{app:fj4} and functions of the form \eqref{app:fj2} or \eqref{app:fj3}, see Section \ref{sec:anotheridentity} for another example.
\end{remark}

\subsection{Proof of Lemma \ref{lem:ConPar}}\label{sec:SpecialCase}
Lemma \ref{lem:ConPar} can be considered as a special case of Theorem \ref{thm:ShiftMaya}.

\begin{proof}[Proof of Lemma \ref{lem:ConPar}]
Take the following Maya diagrams,
\begin{equation} \label{eq:specialM1M2}
	\begin{aligned}
		M_1: & \quad \left( n'_1,\dots,n'_{r_4} \mid \emptyset \right) \\
		M_2: & \quad \left( m'_1,\dots,m'_{r_3} \mid \emptyset \right)
	\end{aligned}
\end{equation} 
where $n'_{r_4}\neq 0$ if $r_4 \geq 1$ and $m'_{r_3}\neq 0$ if $r_3 \geq 0$. Let $\lambda'$ denote the conjugated partition of $\lambda:=\lambda(M_1)$ and similarly for $\mu'$ and $\mu:=\lambda(M_2)$. Set $r(\lambda)$ and $r(\mu)$ as the length of the partitions $\lambda$ and $\mu$.

In the case of \eqref{eq:specialM1M2}, the polynomial \eqref{eq:OmegaM1M2} simplifies to
\begin{equation*}
	\Omega^{(\alpha,\beta)}_{M_1, M_2}(x) 
		= (1+x)^{(\beta+r_3)r_4} (1-x)^{\alpha(r_3+r_4)} \cdot \Wr[f_1,\dots,f_r] 
\end{equation*}
where $r=r_3+r_4$ and $f_1, \ldots, f_r$  are as in \eqref{app:fj3}-\eqref{app:fj4} with $r_1=r_2=0$. As all terms in the Wronskian have a common factor $(1-x)^{-\alpha}$, this factor can be taken outside the Wronskian. In fact, it cancels out the second part of the prefactor. Hence, in terms of definition \eqref{eq:OmegaLaMu}, we directly have  
\begin{equation}\label{eq:conj0}
	\Omega^{(\alpha,\beta)}_{M_1, M_2}
		= \Omega^{(-\alpha,\beta)}_{\mu',\lambda'}.
\end{equation}
If we now consider the monic polynomials and apply \eqref{eq:duality} to the right hand side of \eqref{eq:conj0}, we get
\begin{equation}\label{eq:conj1}
	\widehat{\Omega}^{(\alpha,\beta)}_{M_1, M_2}
		= \widehat{\Omega}^{(-\alpha,-\beta)}_{\lambda',\mu'}.
\end{equation}
Next, we apply Theorem \ref{thm:ShiftMaya} to the left hand side of \eqref{eq:conj1}, we find 
\begin{equation*}
	\widehat{\Omega}^{(\alpha-t_1-t_2,\beta-t_1+t_2)}_{\lambda,\mu}
		= \widehat{\Omega}^{(-\alpha,-\beta)}_{\lambda',\mu'}
\end{equation*}
where
\begin{align*}
	&t_1 = n'_1+1 = \lambda_1+r(\lambda), \\
	&t_2 = m'_1+1= \mu_1+r(\mu).
\end{align*}
This is the result of Lemma \ref{lem:ConPar}, up to a translation of the parameters.

The imposed conditions for the parameters guarantee that both polynomials in \eqref{eq:ConjugatedPartitions} have full degree and moreover, all the above steps are valid for the permitted parameters. This ends the proof.
\end{proof}

\section{Construction of the exceptional Jacobi polynomial}\label{sec:ConstructionXJP}
In this section we revisit exceptional Jacobi polynomials and discuss $X_m$-Jacobi polynomials.

\subsection{Exceptional Jacobi polynomials revisited}\label{sec:XJPrevisited}
We defined the exceptional Jacobi polynomial \eqref{def:XJP1} as a Wronskian, consisting of a fixed set of functions and 1 extra Jacobi polynomial, and an appropriate prefactor. In the light of Section \ref{sec:ConstructionGJP}, we could generalize our definition to
\begin{equation}\label{eq:XJPrevisited}
	c(x) \cdot \Wr[f_1,\dots,f_r,g]
\end{equation}
where
\begin{itemize}
	\item[(a)]
	 the functions $f_1,\dots,f_r$ are as in \eqref{app:fj1}-\eqref{app:fj4}, and not as in \eqref{eq:fj1}-\eqref{eq:fj2},
	 
	 \item[(b)]
	 the function $g$ is an eigenfunction as listed in Table \ref{tab:1}, and not just a Jacobi polynomial; see also \eqref{def:XJP2},
	 
	 \item[(c)]
	 the prefactor $c(x)$ is updated so that we obtain a polynomial.
\end{itemize}
However, for each choice we make in item (b), the polynomial \eqref{eq:XJPrevisited} can be reduced to our definition \eqref{def:XJP1} as described in the following proposition. Recall the convention that $\lambda_1=0$ if $r_1=0$ and $\mu_1=0$ if $r_2=0$.

\begin{proposition}\label{prop:XJPrevisited}
Let $f_1,\dots,f_r$ be as in \eqref{app:fj1}-\eqref{app:fj4} and take two Maya diagrams $M_1,M_2$ as in \eqref{eq:M1M2}. Take $\alpha,\beta\in\mathbb{R}$ such that the conditions \eqref{MayaCondition2GJP} and \eqref{MayaCondition1GJP} are satisfied. Let $\lambda = \lambda(M_1)$ and $\mu = \lambda(M_2)$ be the partitions that are associated with these Maya diagrams and let $r(\lambda)$ and $r(\mu)$ denote the length of these partitions. Let $t_1 = t(M_1)$, $t_2 = t(M_2)$ as in \eqref{eq:tM}.
\begin{enumerate}
	\item[\rm (a)] 
	Take $s \geq 0$ such that $\deg P^{(\alpha,\beta)}_{s} = s$. \\ 
	If $s\neq n_i$ for $i=1,\dots,r_1$, then $n:=s+t_1+|\lambda|+|\mu|-r(\lambda) \in\mathbb{N}_{\lambda,\mu}$ and
	\begin{multline*}
		(1+x)^{(\beta+r_1+1+r_3)(r_2+r_4)}(1-x)^{(\alpha+r_1+1+r_2)(r_3+r_4)}  \\
		\times \Wr[f_1,\dots,f_r,P^{(\alpha,\beta)}_{s}]
		 = C_1 P^{(\alpha-t_1-t_2,\beta-t_1+t_2)}_{\lambda,\mu,n}
	\end{multline*}
	for some constant $C_1\in\mathbb{R}$.
		
	\item[\rm (b)]  
	Take $s \geq 0$ such that $\deg P^{(\alpha,-\beta)}_{s} = s$. \\
	If $s\neq m_i$ for $i=1,\dots,r_2$, then $n:=s+t_2+|\lambda|+|\mu|-r(\mu) \in\mathbb{N}_{\mu,\lambda}$ and  
	\begin{multline*}
		(1+x)^{(\beta+r_1+r_3)(r_2+1+r_4)}(1-x)^{(\alpha+r_1+r_2+1)(r_3+r_4)}  \\
		\times \Wr[f_1,\dots,f_r,(1+x)^{-\beta}P^{(\alpha,-\beta)}_{s}]
		 = C_2 P^{(\alpha-t_1-t_2,-\beta+t_1-t_2)}_{\mu,\lambda,n}
	\end{multline*}
	for some constant $C_2\in\mathbb{R}$.
		
	\item[\rm (c)]  
	Take $s \geq 0$ such that $\deg P^{(-\alpha,\beta)}_{s} = s$. \\
	If $s\neq m'_i$ for $i=1,\dots,r_3$, then $n:=s+t_2+|\lambda|+|\mu|-r(\mu') \in\mathbb{N}_{\mu',\lambda'}$ and 
	\begin{multline*}
		(1+x)^{(\beta+r_1+r_3+1)(r_2+r_4)}(1-x)^{(\alpha+r_1+r_2)(r_3+1+r_4)}  \\
		\times \Wr[f_1,\dots,f_r,(1-x)^{-\alpha}P^{(-\alpha,\beta)}_{s}]
		 = C_3 P^{(-\alpha',\beta')}_{\mu',\lambda',n}
	\end{multline*}
	where $\alpha'= \alpha + (\lambda_1-r_1)+(\mu_1-r_2)$ and $\beta'=\beta + (\lambda_1-r_1) - (\mu_1-r_2)$ and for some constant $C_3\in\mathbb{R}$.
		
	\item[\rm (d)] 
	Take $s \geq 0$ such that $\deg P^{(-\alpha,-\beta)}_{s} = s$. \\
	If $s\neq n'_i$ for $i=1,\dots,r_4$, then $n:=s+t_1+|\lambda|+|\mu|-r(\lambda') \in\mathbb{N}_{\lambda',\mu'}$ and  
	\begin{multline*}
		(1+x)^{(\beta+r_1+r_3)(r_2+r_4+1)}(1-x)^{(\alpha+r_1+r_2)(r_3+r_4+1)}  \\ \times \Wr[f_1,\dots,f_r,(1-x)^{-\alpha}(1+x)^{-\beta}P^{(-\alpha,-\beta)}_{s}]
		 = C_4 P^{(-\alpha',-\beta')}_{\lambda',\mu',n}
	\end{multline*}
	where $\alpha'= \alpha + (\lambda_1-r_1)+(\mu_1-r_2)$ and $\beta'=\beta + (\lambda_1-r_1) - (\mu_1-r_2)$ and for some constant $C_4\in\mathbb{R}$.
\end{enumerate}
Each Wronskian in the above identities can vanish, as we need more conditions for $s$ to obtain that the Wronskian consists of independent entries. In case of a vanishing Wronskian, $C_i=0$. If the Wronskian does not vanish, $C_i \neq 0$.
\end{proposition}
The proof is left out as one simply adapts the proof of Proposition 3 in \cite{Bonneux_Kuijlaars} to the Jacobi setting.

\subsection{$X_m$-Jacobi polynomials}
The first known examples of exceptional Jacobi polynomials were named $X_m$-Jacobi polynomials \cite{GomezUllate_Kamran_Milson-09,GomezUllate_Kamran_Milson-XJacobi,GomezUllate_Marcellan_Milson,Liaw_Littlejohn_Stewart}. Here $m$ is the number of exceptional (or missing) degrees and moreover, these not attained degrees are exactly the first $m$ non-negative integers. As described in \cite{GarciaFerrero_GomezUllate_Milson}, exceptional polynomials are obtained via a series of Darboux transformations on the differential operators for classical orthogonal polynomials. In this sense, $X_m$-Jacobi polynomials are obtained after a specific 1-step Darboux transformation. As this is just a special situation in our setting of two partitions, we identify $X_m$-Jacobi polynomials with our notation, in particular see \eqref{eq:XmId} below.

Let $m$ be a fixed non-negative integer, then the $X_m$-Jacobi polynomials \cite[Equation (64)]{GomezUllate_Marcellan_Milson} are defined as
\begin{multline}\label{eq:XmJac}
	P_{m,n}^{(\alpha,\beta)}(x)
		:= \frac{(-1)^m}{\alpha+1+n-m} \bigg(\frac{1}{2}(1+\alpha+\beta+n-m)(x-1) P_m^{(-\alpha-1,\beta-1)}(x) P_{n-m-1}^{(\alpha+2,\beta)}(x)
		\\
		+ (\alpha+1-m) P_m^{(-\alpha-2,\beta)}(x)P_{n-m}^{(\alpha+1,\beta-1)}(x)\bigg),
		\qquad n\geq m.
\end{multline}
We assume $\alpha>-1$ such that the denominator in \eqref{eq:XmJac} is non-zero, further conditions for the parameters can be found in \cite{GomezUllate_Marcellan_Milson}. A straightforward calculation using \eqref{eq:diff1} and \eqref{eq:diff3} shows that this polynomial can be written as a Wronskian with an appropriate prefactor,
\begin{equation*}
	P_{m,n}^{(\alpha,\beta)}
		= c_1 (1-x)^{\alpha+2} \Wr \left[(1-x)^{-\alpha-1}P_m^{(-\alpha-1,\beta-1)}, P^{(\alpha+1,\beta-1)}_{n-m}\right]
\end{equation*}
for some constant $c_1$. The Wronskian consists of 1 function as in \eqref{app:fj3} and 1 function as in \eqref{app:fj1}. In terms of definition \eqref{eq:OmegaM1M2}, we obtain
\begin{equation}\label{eq:XmJac2}
	P_{m,n}^{(\alpha,\beta)}
		= -c_1 \Omega^{(\alpha,\beta)}_{M_1,M_2}
\end{equation}
where the minus sign is obtained as we have to interchange both functions within the Wronskian, and
\begin{equation*} 
	\begin{aligned}
		M_1: & \quad (\emptyset \mid n-m), \\
		M_2: & \quad (m|\emptyset).
	\end{aligned}
\end{equation*}
We apply Theorem \ref{thm:ShiftMaya} to the right hand side of \eqref{eq:XmJac2}. This implies that the second Maya diagram $M_2$ shifts to $(\emptyset|m,m-1,m-2,\dots,1)$ with corresponding partition $\mu:=\lambda(M_2)=(1,1,\dots,1)$ with $|\mu|=m$, whereas the first Maya diagram $M_1$ is already in its canonical form. Now, as the exceptional Jacobi polynomial \eqref{def:XJP1}  is defined as a Wronskian of a fixed set function and 1 Jacobi polynomial of degree $s$, we can take $\lambda=\emptyset$ and $s=n-m$. We obtain
\begin{equation}\label{eq:XmId}
	P_{m,n}^{(\alpha,\beta)}
		= c_2 P^{(\alpha-m,\beta+m)}_{\emptyset,\mu,n}, \qquad n\in\mathbb{N}_{\emptyset,\mu}=\{m,m+1,m+2,\dots\}
\end{equation}
for some constant $c_2$ and where the degree equals $n$ as $s+|\lambda|+|\mu|-r_1=n$. The shift of the parameters is due to Theorem \ref{thm:ShiftMaya} and the constant changes as Theorem \ref{thm:ShiftMaya} is stated in terms of monic polynomials. We conclude that $X_m$-Jacobi polynomials are exceptional Jacobi polynomials in terms of \eqref{def:XJP1} associated with the empty partition and the partition $(1,1,\dots,1)$ of length $m$ and parameters $\alpha-m$ and $\beta+m$. 

More results for these specific polynomials can be found in \cite{GomezUllate_Marcellan_Milson}. Most of these results concerning the zeros of $X_m$-Jacobi polynomials are generalized in the following section.

\begin{remark}
If we compare the leading coefficients in \eqref{eq:XmId}, we obtain an expression for $c_2$. If $\beta>-1$ and $\beta \neq 0$, we obtain the equality
\begin{equation*}
	(n-m+\alpha+1)(1-n-\beta)_m \prod\limits_{j=1}^{m} (j-2m+\alpha-\beta+1)_j
		= c_2 (-2)^{\frac{m(m-1)}{2}} (m-\alpha+\beta-1)_m (n-2m+\alpha+1).
\end{equation*}
In case where $m=0$, we obtain $c_2=1$ such that \eqref{eq:XmId} reduces to
\begin{equation*}
	P_{0,n}^{(\alpha,\beta)}
		= P^{(\alpha,\beta)}_{\emptyset,\emptyset,n}, \qquad n\geq 0,
\end{equation*}
which are the (classical) Jacobi polynomials. 
\end{remark}

\subsection{Another identity in terms of partitions}\label{sec:anotheridentity}
As mentioned in Remark \ref{rem:type2and3}, the choice to work with the eigenfunctions $(1+x)^{-\beta} P_{m}^{(\alpha,-\beta)}(x)$ is a bit arbitrary. We could also use the eigenfunctions $(1-x)^{-\alpha} P_{m}^{(-\alpha,\beta)}(x)$. In this section, we show the relation between these 2 possible choices. The key idea is that we can shift the second Maya diagram from its canonical choice \eqref{eq:MayaCanonicalFormA} to the form \eqref{eq:MayaCanonicalFormB} while letting the first Maya diagram untouched.

We defined the exceptional Jacobi polynomial as
\begin{equation}\label{eq:P}
	P^{(\alpha,\beta)}_{\lambda,\mu,n}(x) 
		:=  (1+x)^{(\beta+r_1+1)r_2} \cdot \Wr \left[ f_1, \ldots, f_r, P^{(\alpha,\beta)}_s \right]
\end{equation}
where $s = n - |\lambda|- |\mu| + r_1$ with $n\in\mathbb{N}_{\lambda,\mu}$ and
\begin{align*} 
	f_j(x) & = P^{(\alpha,\beta)}_{n_j}(x),  
		&& j=1, \ldots, r_1,  \\
	f_{r_1+j}(x) & = (1+x)^{-\beta}P^{(\alpha,-\beta)}_{m_j}(x), 
		&& j = 1, \ldots, r_2.
\end{align*}
We assumed conditions \eqref{Condition2XJP} and \eqref{Condition1XJP} to be satisfied so that the polynomial has degree $n$.

Similarly, define the polynomials
\begin{align}
	\bar{\Omega}^{(\alpha,\beta)}_{\lambda,\mu,n}(x) 
	&:=  (1-x)^{(\alpha+r_1)r_2} \cdot \Wr \left[ \bar{f}_1, \ldots, \bar{f}_r, \right]
	\label{eq:Omegabar} \\
	\bar{P}^{(\alpha,\beta)}_{\lambda,\mu,n}(x) 
		&:=  (1-x)^{(\alpha+r_1+1)r_2} \cdot \Wr \left[ \bar{f}_1, \ldots, \bar{f}_r, P^{(\alpha,\beta)}_s \right]
	\label{eq:Pbar}	
\end{align}
where $s = n - |\lambda|- |\mu| + r_1$ with $n\in\mathbb{N}_{\lambda,\mu}$ and
\begin{align*} 
	\bar{f}_j(x) & = P^{(\alpha,\beta)}_{n_j}(x),  
		&& j=1, \ldots, r_1,  \\
	\bar{f}_{r_1+j}(x) & = (1-x)^{-\alpha}P^{(-\alpha,\beta)}_{m_j}(x), 
		&& j = 1, \ldots, r_2.
\end{align*}
To be sure that \eqref{eq:Pbar} has degree $n$, we assume the following conditions. The Jacobi polynomials must have full degree,
\begin{equation}\label{Condition2XJPother} 
\begin{aligned}
	&\alpha+\beta+n_i \notin \{-1,-2,\dots,-n_i\},	&&\qquad i=1,\dots,r_1,
	\\
	&\alpha+\beta+s   \notin \{-1,-2,\dots, -s \}, 
	\\
	&-\alpha+\beta+m_i \notin \{-1,-2,\dots,-m_i\},	&&\qquad i=1,\dots,r_2,
\end{aligned}
\end{equation}
and the entries of the Wronskian should be independent,
\begin{equation}\label{Condition1XJPother}
\begin{aligned} 
	&\alpha \neq m_j-n_i, \qquad &&i=1,\dots,r_1 \text{ and } j=1,\dots,r_2, \\
	&\alpha \neq m_j-s,   \qquad &&j=1,\dots,r_2.
\end{aligned} 
\end{equation}

Recall the usual convention $\mu_1=0$ if $r_2=0$ and the notation $\mu'$ for the conjugated partition of $\mu$. We have the following result which gives an identity between \eqref{eq:P} and \eqref{eq:Pbar}.

\begin{lemma}\label{lem:type2and3}
For any partition $\lambda$ and $\mu$, take $\alpha,\beta\in\mathbb{R}$ such that the conditions \eqref{Condition2XJP} and \eqref{Condition1XJP} are satisfied. Moreover, assume that conditions \eqref{Condition2XJPother} and \eqref{Condition1XJPother} are satisfied for the parameters $\alpha'=\alpha+\mu_1+r_2$ and $\beta'=\beta-\mu_1-r_2$ and partitions $\lambda$ and $\mu'$. Then
\begin{equation}\label{eq:type2and3}
	P^{(\alpha,\beta)}_{\lambda,\mu,n}(x) 
		= c \cdot \bar{P}^{(\alpha',\beta')}_{\lambda,\mu',n}(x) 
\end{equation}
for some non-zero constant $c$ as defined in \eqref{eq:constantc}. Moreover,
\begin{equation}\label{eq:type2and3-x}
	P^{(\alpha,\beta)}_{\lambda,\mu,n}(-x) 
	= (-1)^{n+r_1 r_2 + r_1 + r_2}  \bar{P}^{(\beta,\alpha)}_{\lambda,\mu,n}(x).
\end{equation}
\end{lemma}
\begin{proof}
By the assumptions on the parameters and partitions, we know that both polynomials in \eqref{eq:type2and3} have degree $n$. Moreover, note that the last Jacobi polynomial in both Wronskians has the same degree $s$. We now argue that the polynomials are a multiple of each other. 

Consider the second Maya diagram corresponding to the polynomial $P^{(\alpha,\beta)}_{\lambda,\mu,n}$, it is given by	
\begin{equation*}
	\widetilde{M_2} : \quad \left(\emptyset \mid m_1, m_2, \ldots, m_{r_2} \right).
\end{equation*}
Now, we shift this Maya diagram to its other canonical form defined in \eqref{eq:MayaCanonicalFormB}. To do this, we have to shift the origin $m_1+1$ steps to the right. While doing this, the parameters shift in each step as stated in \eqref{eq:reduction2} and \eqref{eq:reduction3}. Hence in total, the parameters $(\alpha,\beta)$ transpose to $(\alpha+\mu_1+r_2,\beta-\mu_1-r_2)$, where we use that $m_1+1=\mu_1+r_2$. The shift to the canonical form \eqref{eq:MayaCanonicalFormB} is such that we end up with the conjugated partition $\mu'$ and the eigenfunctions are now of type \eqref{app:fj3}. Hence we end up with the fact that the monic variant of $P^{(\alpha,\beta)}_{\lambda,\mu,n}$ equals the monic variant of $\bar{P}^{(\alpha+\mu_1+r_2,\beta-\mu_1-r_2)}_{\lambda,\mu',n}$. A constant $c$ is needed to obtain equality in the non-monic case. Comparing leading coefficients of both polynomials gives that
\begin{equation}\label{eq:constantc}
	c
		= \frac
		{
		\frac{\prod\limits_{j=1}^{r_2}(m_j+\alpha-\beta+1)_{m_j}}{2^{ \sum\limits_{j=1}^{r_2}m_j}\prod\limits_{j=1}^{r_2}m_j!}
		\Delta(n_{\mu})\prod\limits_{i=1}^{r_1}\prod\limits_{j=1}^{r_2}(m_j-n_i-\beta) 
		}
		{
		\frac{\prod\limits_{j=1}^{r(\mu')}(m'_j-\alpha'+\beta'+1)_{m'_j}}{2^{ \sum\limits_{j=1}^{r(\mu')}m'_j}\prod\limits_{j=1}^{r(\mu')}m'_j!}
		\Delta(n_{\mu'})\prod\limits_{i=1}^{r_1}\prod\limits_{j=1}^{r(\mu')}(m'_j-n_i-\alpha') 
		}
		\frac{\prod\limits_{j=1}^{r_2}(s+\beta-m_j)}{\prod\limits_{j=1}^{r(\mu')}(s+\alpha'-m_j')}
\end{equation}
where $\alpha'=\alpha+\mu_1+r_2$, $\Delta$ is the Vandermonde determinant, $r(\mu')$ is the length of the partition $\mu'$ and the elements $m'_i$ form the corresponding sequence of the partition $\mu'$. This constant is real and non-zero by the assumptions on the parameters and partitions.

Identity \eqref{eq:type2and3-x} follows from properties \eqref{eq:Wr2} and \eqref{eq:Jacobi-x}. The details are left out.
\end{proof}


\section{Zeros of exceptional Jacobi polynomials: new results}\label{sec:ZerosResults}
In this section we present new results about the zeros of exceptional Jacobi polynomials, the proofs are given in Section \ref{sec:ZerosProofs}. 

As before, let $\lambda$ and $\mu$ be fixed partitions with corresponding sequences $n_{\lambda}=(n_1,\dots,n_{r_1})$ and $n_{\mu}=(m_1,\dots,m_{r_2})$.  Set $r=r_1+r_2$. Assume that the parameters $\alpha,\beta\in\mathbb{R}$ satisfy the conditions \eqref{Condition2XJP} and \eqref{Condition1XJP} such that $\deg P_{\lambda,\mu,n}^{(\alpha,\beta)}=n$; hence the polynomial has $n$ zeros in the complex plane.

\subsection{Number of regular zeros}
The zeros of exceptional polynomials behave different than their classical counterparts: a zero can lie anywhere in the complex plane and the multiplicity of the zero is not necessarily 1.

\begin{definition}
For any partitions $\lambda,\mu$ and $n\in\mathbb{N_{\lambda,\mu}}$, we define $N(n)$ as the number of zeros of the exceptional Jacobi polynomial $P^{(\alpha,\beta)}_{\lambda, \mu,n}$ which lie in the interval $(-1,1)$, where we include multiplicity. We call these zeros the regular zeros of $P^{(\alpha,\beta)}_{\lambda, \mu,n}$. The remaining zeros are called the exceptional zeros.
\end{definition}

We omitted the partitions $\lambda,\mu$ and the parameters $\alpha,\beta$ in our notation of $N(n)$ because it should be clear what they are. The following two results are trivial.
\begin{itemize}
	\item[(a)] 
	Under the conditions of Lemma \ref{lem:XJPdeg}, we have $\deg P_{\lambda,\mu,n}^{(\alpha,\beta)}=n$ such that $0\leq N(n) \leq n$ and the number of exceptional zeros is given by $n-N(n)$.
	
	\item[(b)] 
	Under the conditions of Lemma \ref{lem:XJPorth}, the number $N(n)$ can be computed explicitly by standard Sturm-Liouville theory:
	\begin{equation}\label{eq:NumberRegularZeros}
		N(n)
			= |\{m \in \mathbb{N_{\lambda,\mu}} : m<n\}|, \qquad n\in\mathbb{N}_{\lambda,\mu}.
	\end{equation}
	Moreover, all these regular zeros are simple. Hence, for $n$ large enough, we have $n-|\lambda|-|\mu|$ simple regular zeros and $|\lambda|+|\mu|$ exceptional zeros. 
\end{itemize}

Inspecting \eqref{eq:NumberRegularZeros} gives that the number of exceptional zeros is bounded when the degree $n$ tends to infinity. This result holds true in a more general setting as presented in the following theorem. The proof is given in Section \ref{sec:ProofN(n)}. 

\begin{theorem} \label{thm:XJPN(n)}
For any partitions $\lambda$ and $\mu$, take $\alpha,\beta\in\mathbb{R}$ and $n\in\mathbb{N}_{\lambda,\mu}$ such that the conditions \eqref{Condition2XJPbis} and \eqref{Condition1XJP} are satisfied. If $\alpha+r>-1$ and $\beta+r>-1$, then
\begin{equation*}
	n-2(|\lambda|+|\mu|+r_2) \leq N(n).
\end{equation*}
Moreover, the number of simple regular zeros increases to infinity as $n$ tends to infinity.
\end{theorem}

The theorem does not say anything about the exact number of regular zeros. In case where $r_1=0$ or $r_2=0$ and $\alpha\neq\beta$, this number can be determined explicitly for almost all $\alpha$ and $\beta$, see \cite{GarciaFerrero_GomezUllate}. It is given by the alternating sum of the elements in the partition. 

Via Theorem \ref{thm:XJPN(n)}, we describe the asymptotic behavior of the zeros for exceptional Jacobi polynomial. For the regular zeros we obtain that they have the same asymptotic behavior as the zeros for classical Jacobi polynomials, see Corollary \ref{cor:XJPRegularZeros} and Theorem \ref{thm:XJPArcsineLaw}. The exceptional zeros converge to the zeros of the generalized Jacobi polynomial, see Theorem \ref{thm:XJPExceptionalZeros}. These results justify the conjecture in \cite{Kuijlaars_Milson} for exceptional Jacobi polynomials.

\subsection{Mehler-Heine asymptotics}
Mehler-Heine asymptotics describe the asymptotic behavior near the endpoints of the orthogonality region for orthogonal polynomials. For Jacobi polynomials, the result is stated in Theorem \ref{thm:JacMehlerHeine} below. In the exceptional setting, a similar result holds true as described in the following theorem. The proof is given in Section \ref{sec:ProofMehlerHeine}. Let $J_{\nu}$ denote the Bessel function of the first kind with parameter $\nu\in\mathbb{R}$.

\begin{theorem}\label{thm:XJPMehlerHeine}
For any partitions $\lambda$ and $\mu$, take $\alpha,\beta\in\mathbb{R}$ such that the conditions \eqref{Condition2GJP} and \eqref{Condition1GJP} are satisfied. Let $r(\mu')$ the length of the conjugated partition $\mu'$.

\begin{itemize}
	\item[(a)]
	Set $r=r_1+r_2$, then
	\begin{equation}\label{eq:XJPMehlerHeineAlpha}
		\lim_{n \to \infty} \frac{1}{n^{\alpha+2r}} P^{(\alpha,\beta)}_{\lambda, \mu,n} \left( \cos\left(\frac{x}{n}\right) \right)
			= \Omega^{(\alpha,\beta)}_{\lambda,\mu}(1) \, 2^{\alpha+r_2} x^{-\alpha-r} J_{\alpha + r}(x),
	\end{equation}
	uniformly for $x$ in compact subsets of the complex plane.
	
	\item[(b)]
	If the conditions \eqref{Condition2XJPother} and \eqref{Condition1XJPother} are satisfied for the parameters $\alpha'=\alpha+\mu_1+r_2$ and $\beta'=\beta-\mu_1-r_2$ and partitions $\lambda$ and $\mu'$, then
	\begin{equation}\label{eq:XJPMehlerHeineBeta}
		\lim_{n \to \infty} \frac{(-1)^{n-|\lambda|-|\mu|+r(\mu')}}{n^{\beta+2r_1}} P^{(\alpha,\beta)}_{\lambda, \mu,n} \left( -\cos\left(\frac{x}{n}\right) \right)
		= d \cdot \bar{\Omega}^{(\alpha',\beta')}_{\lambda,\mu'}(-1) \, x^{-\beta-r_1+r_2}J_{\beta + r_1-r_2}(x),
	\end{equation}
	uniformly for $x$ in compact subsets of the complex plane and for some non-zero $d\in\mathbb{R}$, and where $\bar{\Omega}^{(\alpha',\beta')}_{\lambda,\mu'}$ is defined in \eqref{eq:Omegabar}.
\end{itemize}
\end{theorem}

This asymptotic behavior allows us to describe a convergence property for the regular zeros near the endpoints $-1$ and $1$ if the value $\Omega^{(\alpha,\beta)}_{\lambda,\mu}(1)$ (respectively $\bar{\Omega}^{(\alpha',\beta')}_{\lambda,\mu'}(-1)$) is non-zero. We know that the function $x^{-\nu}J_{\nu}(x)$ is an entire function in the complex plane with an infinite number of zeros on the real line, which are all simple, except possibly at the origin. Moreover, the function has a zero at the origin if and only if $\nu\leq -1$. Therefore, if we apply Hurwitz theorem \cite[Theorem 1.91.3]{Szego} on \eqref{eq:XJPMehlerHeineAlpha} or \eqref{eq:XJPMehlerHeineBeta}, we get that the regular zeros near the edge converge to the zeros of the Bessel function. Let $j_{\nu,k}$ denote the $k^\text{th}$ positive zero of the Bessel function $J_{\nu}$.

\begin{corollary} \label{cor:XJPRegularZeros}
For any partitions $\lambda$ and $\mu$, take $\alpha,\beta\in\mathbb{R}$ such that the conditions \eqref{Condition2GJP} and \eqref{Condition1GJP} are satisfied. Assume that $\alpha+r>-1$ and $\beta+r>-1$. Then for every positive integer $k$, there is an integer $n_0$ such that for all $n\in\mathbb{N}_{\lambda,\mu}$ with $n\geq n_0$, the exceptional Jacobi polynomial $P^{(\alpha,\beta)}_{\lambda, \mu,n}$ has at least $k$ zeros in $(-1,1)$. Set $\alpha'=\alpha+\mu_1+r_2$ and $\beta'= \beta-\mu_1-r_2$.
\begin{enumerate}
	\item[\rm (a)] 
	Let $x_{k,n}^{(\alpha,\beta)}$ denote the $k^{\text{th}}$ regular zero which is the $k^{\text{th}}$ closest zero to 1, including multiplicity. Take the unique $0 \leq \theta^{(\alpha,\beta)}_{k,n} \leq \pi$ such that $\cos\left(\theta^{(\alpha,\beta)}_{k,n} \right)=x_{k,n}^{(\alpha,\beta)}$. 
	If $\Omega^{(\alpha,\beta)}_{\lambda,\mu}(1)\neq 0$, then we have
	\begin{equation*}
	\lim_{n \to \infty} n \theta^{(\alpha,\beta)}_{k,n}  = j_{\alpha+r_1+r_2,k}.
	\end{equation*}
	
	\item[\rm (b)] 
	Assume that the conditions \eqref{Condition2XJPother} and \eqref{Condition1XJPother} are satisfied for the parameters $\alpha'=\alpha+\mu_1+r_2$ and $\beta'=\beta-\mu_1-r_2$ and partitions $\lambda$ and $\mu'$. Let $y_{k,n}^{(\alpha,\beta)}$ denote the $k^{\text{th}}$ regular zero which is the $k^{\text{th}}$ closest zero to -1, including multiplicity. Take the unique $0 \leq \Theta^{(\alpha,\beta)}_{k,n} \leq \pi$ such that $\cos\left(\Theta^{(\alpha,\beta)}_{k,n} \right)=-x_{k,n}^{(\alpha,\beta)}$. 
	If $\bar{\Omega}^{(\alpha',\beta')}_{\lambda,\mu'}(-1)\neq 0$, then we have
	\begin{equation*}
	\lim_{n \to \infty} n \Theta^{(\alpha,\beta)}_{k,n}  = j_{\beta+r_1-r_2,k}.
	\end{equation*}
\end{enumerate}
\end{corollary}
\begin{proof}
This corollary follows directly from applying Hurwitz theorem as described above. Note that we also need the conditions of Theorem \ref{thm:XJPN(n)} to be satisfied.
\end{proof}

\subsection{Weak macroscopic limit of the regular zeros}
When $\alpha>-1$ and $\beta>-1$, the weak macroscopic limit of the simple zeros of the classical Jacobi polynomial is the arcsine distribution $\frac{1}{\pi} \frac{1}{\sqrt{1-x^2}}dx$ (or arcsine law), see Theorem \ref{thm:JacArcsineLaw} below. For exceptional Jacobi polynomials, we know from Theorem \ref{thm:XJPN(n)} that the number of regular zeros tends to infinity when the degree tends to infinity. If we take the weak macroscopic limit of those regular zeros, it turns out that the same limiting distribution as for classical Jacobi polynomials is obtained. The proof is given in Section \ref{sec:ProofArcsineLaw}.

\begin{theorem} \label{thm:XJPArcsineLaw}
For any partition $\lambda$ and $\mu$, take $\alpha,\beta\in\mathbb{R}$ such that the conditions \eqref{Condition2GJP} and \eqref{Condition1GJP} are satisfied. Assume that $\alpha+r>-1$ and $\beta+r>-1$. Let $1>x^{(\alpha,\beta)}_{1,n}\geq\cdots\geq x^{(\alpha,\beta)}_{N(n),n}>-1$ denote the regular zeros of the exceptional Jacobi polynomial $P^{(\alpha,\beta)}_{\lambda, \mu,n}$ where $n\in\mathbb{N}_{\lambda,\mu}$. Then, for every bounded continuous function $f$ on $(-1,1)$, we have
\begin{equation} \label{eq:XJPArcsineLaw}
	\lim_{n\to\infty}\frac{1}{N(n)}\sum_{j=1}^{N(n)}f\left(x^{(\alpha,\beta)}_{j,n}\right)
		=\frac{1}{\pi}\int_{-1}^{1} \frac{f(x)}{\sqrt{1-x^2}}dx.
\end{equation}
\end{theorem}

\subsection{Convergence of the exceptional zeros}
Our final result deals with the asymptotic behavior of the exceptional zeros. To this end, recall that the number of exceptional zeros is bounded as the degree tends to infinity, see Theorem \ref{thm:XJPN(n)}. The theorem states that these exceptional zeros are attracted by the simple zeros of the generalized Jacobi polynomial.

\begin{theorem} \label{thm:XJPExceptionalZeros}
For any partitions $\lambda$ and $\mu$, take $\alpha,\beta\in\mathbb{R}$ such that the conditions \eqref{Condition2GJP} and \eqref{Condition1GJP} are satisfied. Assume that $\alpha+r>-1$ and $\beta+r>-1$. Let $z_j$ be a simple zero of the generalized Jacobi polynomial $\Omega^{(\alpha,\beta)}_{\lambda,\mu}$ where $z_j\in\mathbb{C}\setminus[-1,1]$. Then this zero $z_j$ attracts an exceptional zero of the exceptional Jacobi polynomial $P^{(\alpha,\beta)}_{\lambda,\mu,n}$ as $n$ tends to infinity at a rate $O\left(n^{-1}\right)$.
That is, for $n$ large enough, we have 
\begin{equation} \label{eq:XJPExceptionalZeros}
	\min_{k = 1, \ldots, n-N(n)} 
	\left| z_j-z^{(\alpha,\beta)}_{k,n} \right| < \frac{c}{n}, \qquad n\in\mathbb{N}_{\lambda,\mu}
\end{equation}
for some positive constant $c$ and where $z^{(\alpha,\beta)}_{1,n},\dots,z^{(\alpha,\beta)}_{n-N(n),n}$ denote the exceptional zeros of the exceptional Jacobi polynomial $P^{(\alpha,\beta)}_{\lambda,\mu,n}$.
\end{theorem}

\begin{figure}[t]
	\centering
	\includegraphics[totalheight=8cm, trim=2.45cm 17.5cm 8.75cm 2cm,clip=true]{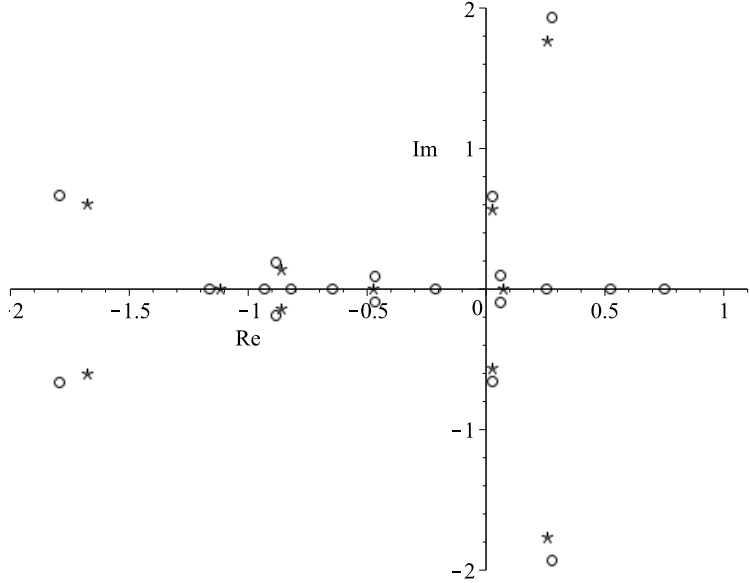}
	\caption{Zeros of the generalized (stars) and exceptional (open circles) Jacobi polynomial associated with $\lambda=(3,1,1)$, $\mu=(3,3)$, $\alpha=0$, $\beta=\frac{1}{2}$ and $n=20$.}
	\label{fig:1}
\end{figure}

\begin{remark}
Similar results for Theorem \ref{thm:XJPExceptionalZeros} are obtained in the Hermite case \cite[Theorem 2.3]{Kuijlaars_Milson} and Laguerre case \cite[Theorem 5]{Bonneux_Kuijlaars}. A remarkable difference is that in the Jacobi case, the rate of convergence is faster, i.e., $O\left(n^{-1}\right)$, where for the Hermite and Laguerre case it is $O\left(n^{-\frac{1}{2}}\right)$. The difference in the results is due to the fact that all zeros of the Jacobi polynomial lie in the bounded region $(-1,1)$ whereas for Hermite and Laguerre polynomials the region is unbounded. For all three results, the speed of convergence is not claimed to be sharp. However, simulations by the computer software Maple suggest that all bounds are in fact sharp.
\end{remark}

As a special case of Theorem \ref{thm:XJPExceptionalZeros}, consider the situation where the exceptional orthogonal polynomials form a complete system, see Lemma \ref{lem:XJPorth}. Then $N(n)=n-|\lambda|-|\mu|$ for $n\in\mathbb{N}_{\lambda,\mu}$ large enough so that there are $|\lambda|+|\mu|$ exceptional zeros. Moreover, all zeros of the generalized Jacobi polynomial are in $\mathbb{C}\setminus[-1,1]$, see Lemma \ref{lem:OmegaZerosOrth}. Therefore, if we assume all zeros of the generalized Jacobi polynomial to be simple, then each zero attracts exactly 1 exceptional zero. If we label the simple zeros by $z_j$, then we can relabel the zeros of the exceptional Jacobi polynomial in such a way that $z_{j,n}^{(\alpha,\beta)}$ is close to $z_j$ and
\begin{equation*}
	z_{j,n}^{(\alpha,\beta)} = z_j + O\left(\frac{1}{n}\right) \qquad
	\text{ as } n \to \infty.
\end{equation*}
Moreover, if $z_j\in\mathbb{R}\setminus[-1,1]$, then $z^{(\alpha,\beta)}_{j,n}$ must be real and outside the interval $(-1,1)$ for $n$ large enough; otherwise it would be a regular zero or if it was non-real, its complex conjugate would be a zero too and hence $z_j$ would attract two zeros which is prohibited. \\

Our results are numerically verified via the computer software Maple. In Figure \ref{fig:1} we plotted an example where the zeros of the generalized Jacobi polynomial are all simple. As we can see in the figure, each exceptional zero is close to a zero of the generalized Jacobi polynomial. Moreover, the zeros of the generalized Jacobi polynomial in $(-1,1)$ seem to attract two exceptional zeros.

\subsection{Conjecture of simple zeros}
It is well-known that the zeros of the Jacobi polynomial are simple when the parameters satisfy $\alpha,\beta>-1$. For generalized Jacobi polynomials, which include exceptional Jacobi polynomial, this property of simple zeros is far from trivial. According to Proposition 3.3 in \cite{Duistermaat_Grunbaum}, the multiplicity of each zero is a triangular number. The following examples show that for general partitions and parameters, the multiplicity is not necessarily 1, neither inside nor outside the orthogonality region $(-1,1)$.
\begin{align*}
	& \Omega^{(1,1)}_{(1,1),(1)} = -15(x+1)^3
	&& \Omega^{(\frac{5}{2},-\frac{3}{2})}_{(2),(2)} = -\frac{105}{128}(4x+5)(2x+1)^3
	\\
	& \Omega^{(\frac{9}{2},\frac{9}{2})}_{(1,1),(1)} = - \frac{5005}{8} x^3
	&& \Omega^{(\frac{1}{2},-\frac{1}{2})}_{(2),(4)} = \frac{945}{2048} \left(x-\left(\frac{1+\sqrt{5}}{4}\right)\right)^3\left(x-\left(\frac{1-\sqrt{5}}{4}\right)\right)^3 
\end{align*}
All the examples we found, dealing with non-simple zeros, are such that the corresponding exceptional Jacobi polynomials do not form a complete set. We therefore offer the following conjecture which suggests that the assumption in Theorem \ref{thm:XJPExceptionalZeros} dealing with simple zeros of the generalized Jacobi polynomial is not a restriction.

\begin{conjecture}\label{conjecture}
For any partitions $\lambda$ and $\mu$, take $\alpha>-1$ and $\beta > m_1$ such that the conditions \eqref{Condition1GJP} are satisfied. If $\lambda$ is an even partition, then the zeros of the generalized Jacobi polynomial $\Omega^{(\alpha,\beta)}_{\lambda,\mu}$ are simple. 
\end{conjecture}

This conjecture is closely related to the Veselov conjecture \cite{Felder_Hemery_Veselov} dealing with simple zeros of the Wronskian of an arbitrary (finite) sequence of Hermite polynomials, and with Conjecture 1 in \cite{Bonneux_Kuijlaars} for the Laguerre case.

\section{Zeros of exceptional Jacobi polynomials: proofs}\label{sec:ZerosProofs}
In this section we give the proofs of the results presented in Section \ref{sec:ZerosResults}. Similar ideas as in the Hermite \cite{Kuijlaars_Milson} and Laguerre \cite{Bonneux_Kuijlaars} cases are used. 

\subsection{Proof of Theorem \ref{thm:XJPN(n)}: a lower bound of $N(n)$}\label{sec:ProofN(n)}
We start with the following lemma for Jacobi polynomials.

\begin{lemma}\label{lem:JacobiShiftParameter}
Take $\alpha,\beta\in\mathbb{R}$ and two non-negative integers $n$ and $N$ such that 
\begin{equation}\label{eq:JacobiShiftParameterRestriction}
	\alpha+\beta+n\notin\{-1,-2,\dots,-n-2N\}.
\end{equation}
Then the Jacobi polynomial $P_n^{(\alpha,\beta)}$ is a linear combination of the Jacobi polynomials
\begin{equation*}
	P_{n}^{(\alpha+N,\beta+N)}, P_{n-1}^{(\alpha+N,\beta+N)}, \dots, P_{n-t}^{(\alpha+N,\beta+N)}
\end{equation*} 
where $t=2N$ is independent of $n$.
\end{lemma}
\begin{proof}
Fix $\alpha,\beta\in\mathbb{R}$ and $n,N$ such that \eqref{eq:JacobiShiftParameterRestriction} is satisfied. Then,
\begin{align*}
	&\deg \left(P_n^{(\alpha,\beta)}(x)\right) = n, \\
	&\deg \left(P^{(\alpha+N,\beta+N)}_{i}(x)\right) = i, \qquad i=0,\dots,n.
\end{align*}	
Hence we may write
\begin{equation*}
	P_n^{(\alpha,\beta)}(x)
		= \sum_{i=0}^{n} c^{\alpha,\beta,N}_{i,n}  P^{(\alpha+N,\beta+N)}_{i}(x)
\end{equation*}
for some $c^{\alpha,\beta,N}_{i,n}\in\mathbb{R}$. Now we have to show that $c^{\alpha,\beta,N}_{i,n}=0$ when $i<n-2N$. We approach by first assuming $\alpha,\beta>-1$ and then considering general $\alpha,\beta\in\mathbb{R}$.

If $\alpha,\beta>-1$, then the results follow directly from the orthogonality of the Jacobi polynomials. To be precise, the orthogonality of the Jacobi polynomials $\left(P^{(\alpha+N,\beta+N)}_{k}\right)_{k}$ gives
\begin{align*}
	c^{\alpha,\beta,N}_{i,n}
		&= \frac{1}{\gamma^{2}_i} \int_{-1}^{1} P_n^{(\alpha,\beta)}(x) P^{(\alpha+N,\beta+N)}_{i}(x) (1-x)^{\alpha+N} (1+x)^{\beta+N} dx \\
		&= \frac{1}{\gamma^{2}_i} \int_{-1}^{1}(1-x^2)^N P^{(\alpha+N,\beta+N)}_{i}(x)  P_n^{(\alpha,\beta)}(x) (1-x)^{\alpha} (1+x)^{\beta} dx
\end{align*}
for some normalization constant $\gamma^{2}_i\in\mathbb{R}\setminus\{0\}$ (as we do not have orthonormality). Next, we use the orthogonality of the Jacobi polynomials $\left(P^{(\alpha,\beta)}_{k}\right)_{k}$ to conclude that $c^{\alpha,\beta,N}_{i,n}=0$ if $2N+i<n$. This argument is also used in the proof of Theorem 3.21 in \cite{Shen_Tang_Wang}.

For general parameters $\alpha$ and $\beta$, the constants $c^{\alpha,\beta,N}_{i,n}$ can be computed exactly and can be expressed via the generalized hypergeometric functions, see \cite[Equation (7.28)]{Askey},
\begin{multline*}
	c^{\alpha,\beta,N}_{i,n}
		= \frac{\Gamma(i+\alpha+\beta+2N+1)\Gamma(n+i+\alpha+\beta+1)\Gamma(n+\alpha+1)}{\Gamma(n+\alpha+\beta+1)\Gamma(i+\alpha+1)\Gamma(2i+\alpha+\beta+2N+1)(n-i)!} \\
		\pFq{3}{2}{-n+i,n+i+\alpha+\beta+1,i+\alpha+N+1}{i+\alpha+1,2i+\alpha+\beta+2N+2}{1},
\end{multline*}
Hence we can interpret these constants as meromorphic functions in the variables $\alpha,\beta\in\mathbb{C}$. When these variables are real such that $\alpha,\beta>-1$, we already found by orthogonality that they equal zero when $i<n-2N$. As a direct consequence, the meromorphic function $c^{\alpha,\beta,N}_{i,n}$ is zero everywhere when $i<n-2N$. This establishes the result. 
\end{proof}

A similar result holds true for exceptional Jacobi polynomials. Observe that we need conditions \eqref{Condition2XJPbis} to be satisfied (and not just conditions \eqref{Condition2XJP}) as we want to apply Lemma \ref{lem:JacobiShiftParameter}.

\begin{lemma}\label{lem:XJPlincomJac}
For any partitions $\lambda$ and $\mu$, take $\alpha,\beta\in\mathbb{R}$ and $n\in\mathbb{N}_{\lambda,\mu}$ such that $\alpha+r>-1$, $\beta+r>-1$ and the conditions \eqref{Condition2XJPbis} and \eqref{Condition1XJP} are satisfied. Then, the exceptional Jacobi polynomial $P^{(\alpha,\beta)}_{\lambda,\mu,n}$ is a linear combination of the Jacobi polynomials  
\begin{equation*}
	P^{(\alpha+r,\beta+r)}_{n}, P^{(\alpha+r,\beta+r)}_{n-1},\dots,P^{(\alpha+r,\beta+r)}_{n-t}
\end{equation*}
where $t=2(|\lambda|+|\mu|+r_2)$ is independent of $n$. 
\end{lemma}
\begin{proof}
Let $s=n-|\lambda|-|\mu|+r_1$ and expand the exceptional Jacobi polynomial \eqref{def:XJP1} by its last column, it gives
\begin{equation}\label{eq:XJPlincomJac:1}
	P^{(\alpha,\beta)}_{\lambda,\mu,n}(x) 
		= \sum_{k=0}^{r}Q_{k}(x)\frac{d^k}{dx^k}P^{(\alpha,\beta)}_{s}(x)
\end{equation}	
for some polynomials $Q_k$ as described in Proposition \ref{prop:degQ}. Furthermore, the same proposition gives
\begin{equation}\label{eq:XJPlincomJac:1.5}
	\deg Q_k \leq |\lambda|+|\mu|-r_1+k.
\end{equation} 
Note that the assumptions of the proposition coincide with conditions \eqref{Condition1XJP}. 

The $k^{\text{th}}$ derivative of the Jacobi polynomial is given in \eqref{eq:JacobiDerivative}, hence \eqref{eq:XJPlincomJac:1} can be written as
\begin{equation}\label{eq:XJPlincomJac:2}
	P^{(\alpha,\beta)}_{\lambda,\mu,n}(x)
		= \sum_{k=0}^{r} \frac{(n+\alpha+\beta+1)_{k}}{2^k} Q_{k}(x) P^{(\alpha+k,\beta+k)}_{s-k}(x).
\end{equation}
It is sufficient if $k$ runs from zero to $\min\{r,s\}$, then the subindex $s-k\geq0$ for all $k$.

Next, we express all Jacobi polynomials in terms Jacobi polynomials with parameters $\alpha+r$ and $\beta+r$ via Lemma \ref{lem:JacobiShiftParameter}. The assumptions for this lemma are satisfied as we impose conditions \eqref{Condition2XJPbis}, we get
\begin{equation}\label{eq:JacobiShiftParameter}
	P^{(\alpha,\beta)}_{s-k}
		= \sum_{i=0}^{2(r-k)} c^{\alpha,\beta,r-k}_{s-k-i,s-k} P^{(\alpha+r,\beta+r)}_{s-k-i}(x)
\end{equation} 
where $c^{\alpha,\beta,r-k}_{s-k-i,s-k}\in\mathbb{R}$ for all $i$. If we combine \eqref{eq:XJPlincomJac:2} and \eqref{eq:JacobiShiftParameter}, we obtain
\begin{equation*}
	P^{(\alpha,\beta)}_{\lambda,\mu,n}(x) 
		= \sum_{k=0}^{r} \frac{(n+\alpha+\beta+1)_{k}}{2^k} Q_{k}(x) \sum_{i=0}^{2(r-k)} c^{\alpha,\beta,r-k}_{s-k-i,s-k} P^{(\alpha+r,\beta+r)}_{s-k-i}(x).
\end{equation*}
Rearranging the terms in the right hand side gives
\begin{equation}\label{eq:XJPlincomJac:3}
	P^{(\alpha,\beta)}_{\lambda,\mu,n}(x)
		= 
		\sum_{j=0}^{r}\tilde{Q}_{j}(x)P^{(\alpha+r,\beta+r)}_{s-j}(x) 
		+ 
		\sum_{j=1}^{r}\hat{Q}_{r-j}(x)P^{(\alpha+r,\beta+r)}_{s-r-j}(x)
\end{equation}
where $\tilde{Q}_j$ and $\hat{Q}_j$ are polynomials of degree at most $|\lambda|+|\mu|-r_1+j$ for $j=0,\dots,r$; this follows directly from \eqref{eq:XJPlincomJac:1.5}. Since we assume $\alpha+r>-1$ and $\beta+r>-1$, we can apply the three-term recurrence relation for the Jacobi polynomials in \eqref{eq:XJPlincomJac:3} as they are orthogonal. We get
\begin{align}
	\tilde{Q}_{j}(x)P^{(\alpha+r,\beta+r)}_{s-j}(x)
		&= \sum_{m=-\deg \tilde{Q}_{j}}^{\deg \tilde{Q}_{j}} a^j_m P^{(\alpha+r,\beta+r)}_{s-j+m}(x), 
	&& 0\leq j \leq r, \label{eq:XJPlincomJac:4} \\
	\hat{Q}_{r-j}(x)P^{(\alpha+r,\beta+r)}_{s-r-j}(x)
		&= \sum_{m=-\deg \hat{Q}_{r-j}}^{\deg \hat{Q}_{r-j}} b^j_m P^{(\alpha+r,\beta+r)}_{s-r-j+m}(x), 
	&& 1\leq j \leq r \label{eq:XJPlincomJac:5},	
\end{align}
where $a^j_m,b^j_m\in\mathbb{R}$. Hence, inspecting all terms in \eqref{eq:XJPlincomJac:3}, we get via \eqref{eq:XJPlincomJac:4} and \eqref{eq:XJPlincomJac:5} that we end up with a linear combination of Jacobi polynomials $P^{(\alpha+r,\beta+r)}_{M}$ where the degree $M$ is in the range $n-2(|\lambda|+|\mu|+r_2)\leq M \leq n$. Hence the result is established.
\end{proof}

By the result of Lemma \ref{lem:XJPlincomJac}, we prove Theorem \ref{thm:XJPN(n)} using the orthogonality of the Jacobi polynomials. The argument is completely the same as the proof of Theorem 2 in \cite{Bonneux_Kuijlaars}.

\begin{proof}[Proof of Theorem \ref{thm:XJPN(n)}]
Since $\alpha+r>-1$ and $\beta+r>-1$, the Jacobi polynomials $P^{(\alpha+r,\beta+r)}_{n}$ are orthogonal polynomials. According to Lemma \ref{lem:XJPlincomJac}, the exceptional Jacobi polynomial is a linear combination of these Jacobi polynomials. Hence for every polynomial $H$ of degree strict less than $n-t$, we get
\begin{equation*}
	\int_{-1}^{1}H(x) P_{\lambda,\mu,n}^{(\alpha,\beta)}(x) (1+x)^{\alpha+r}(1-x)^{\beta+r}dx=0.
\end{equation*}
This implies that $P_{\lambda,\mu}^{(\alpha,\beta,n)}$ has at least $n-t$ zeros in $(-1,1)$; by definition these zeros are regular zeros. Moreover, these zeros must have odd multiplicity. Hence the number of simple zeros must tend to infinity when the degree tends to infinity.
\end{proof} 

Lemma \ref{lem:XJPlincomJac} has another direct consequence which we need further. 

\begin{corollary}\label{cor:BeardonDriver}
For any partitions $\lambda$ and $\mu$, take $\alpha,\beta\in\mathbb{R}$ and $n\in\mathbb{N}_{\lambda,\mu}$ such that the conditions \eqref{Condition2XJPbis} and \eqref{Condition1XJP} are satisfied. Assume that $\alpha+r>-1$ and $\beta+r>-1$ and let $n>2(|\lambda|+|\mu|+r_2)$. Further, denote by 
$1>a^{(\alpha+r,\beta+r)}_{1,n}>a^{(\alpha+r,\beta+r)}_{2,n}>\cdots>a^{(\alpha+r,\beta+r)}_{n,n}>-1$ the simple zeros of $P^{(\alpha+r,\beta+r)}_{n}$. Then, at least $n-2(|\lambda|+|\mu|+r_2)$ intervals $\left(a^{(\alpha+r,\beta+r)}_{j,n},a^{(\alpha+r,\beta+r)}_{j+1,n}\right)$, where $1\leq j<n$, contain a (regular) zero of the exceptional Jacobi polynomial $P^{(\alpha,\beta)}_{\lambda,\mu,n}$.
\end{corollary}

This corollary is a special case of the result of Beardon and Driver, see \cite[Theorem 3.2]{Beardon_Driver}. It basically says that almost all regular zeros follow the zeros of the Jacobi polynomial when the degree of the exceptional Jacobi polynomial tends to infinity.

\subsection{Proof of Theorem \ref{thm:XJPMehlerHeine}: Mehler-Heine asymptotics}\label{sec:ProofMehlerHeine}
We recall the Mehler-Heine asymptotics for Jacobi polynomials \cite[Theorem 8.1.1]{Szego}. In this theorem, \eqref{eq:JacMehHeiBeta} directly follows from \eqref{eq:JacMehHeiAlpha} via \eqref{eq:Jacobi-x}.

\begin{theorem}\label{thm:JacMehlerHeine}
Take $\alpha,\beta\in\mathbb{R}$, then
\begin{align}
	&\lim_{n\to\infty} \frac{1}{n^{\alpha}} P_n^{(\alpha,\beta)}\left(\cos\left(\frac{x}{n}\right)\right) 	
		= 2^{\alpha} x^{-\alpha} J_{\alpha}\left(x\right), 
	\label{eq:JacMehHeiAlpha}\\
	&\lim_{n\to\infty} \frac{(-1)^n}{n^{\beta}} P_n^{(\alpha,\beta)}\left(-\cos\left(\frac{x}{n}\right)\right) 
		= 2^{\beta} x^{-\beta} J_{\beta}\left(x\right),
	\label{eq:JacMehHeiBeta}
\end{align}
uniformly for $x$ in compact subsets of the complex plane.
\end{theorem}	

Now we prove a similar behavior for exceptional Jacobi polynomials.

\begin{proof}[Proof of Theorem \ref{thm:XJPMehlerHeine}]
We start with the proof of \eqref{eq:XJPMehlerHeineAlpha}. We obtain this result by extending \eqref{eq:JacMehHeiAlpha}. 

Set $s = n-|\lambda|-|\mu|+r_1$ and consider the decomposition of $P^{(\alpha,\beta)}_{\lambda,\mu,n}(x)$ 
as in \eqref{eq:XJPlincomJac:1}. We have $Q_r = (1+x)^{r_2} \Omega_{\lambda,\mu}^{(\alpha,\beta)}$
and therefore \eqref{eq:XJPlincomJac:1} can be written as
\begin{equation}\label{eq:MehlerHeine:1}
	P^{(\alpha,\beta)}_{\lambda,\mu,n}(x) 
		= \sum_{j=0}^{r-1}Q_{j}(x)\frac{d^j}{dx^j}P^{(\alpha,\beta)}_{s}(x) + (1+x)^{r_2}
		\Omega^{(\alpha,\beta)}_{\lambda,\mu}(x) \frac{d^r}{dx^r}P^{(\alpha,\beta)}_{s}(x).
\end{equation}
The degree of the polynomial $Q_j$ is at most $|\lambda|+|\mu|-r_1+j$, which is independent of $n$, see Proposition \ref{prop:degQ}. Hence, only the Jacobi polynomials in \eqref{eq:MehlerHeine:1} depend on $n$ (recall that $s$ depends on $n$).

The limit \eqref{eq:JacMehHeiAlpha} also holds if we replace
$P_n^{(\alpha,\beta)}$ by $P_{n-c}^{(\alpha,\beta)}$ for some constant $c$. In particular, for $c=|\lambda|+|\mu|-r_1$, we get that $s=n-c$, and hence
\begin{equation}\label{eq:MehlerHeine:2}
	\lim_{n\to\infty} \frac{1}{n^{\alpha}}P_s^{(\alpha,\beta)}\left(\cos\left(\frac{x}{n}\right)\right) 
		= 2^{\alpha} x^{-\alpha} J_{\alpha}\left(x\right)
\end{equation}
uniformly for $x$ in compact subsets of the complex plane.	
Because of the uniform convergence, \eqref{eq:MehlerHeine:2} can be differentiated 
with respect to $x$ any number of times. If we differentiate 1 time, we obtain 
\begin{equation}\label{eq:firstderivative}
	\lim_{n\to\infty} \frac{-1}{n^{\alpha+1}} \sin\left(\frac{x}{n}\right) \frac{d}{dx}\left(P_s^{(\alpha,\beta)}\right)\left(\cos\left(\frac{x}{n}\right)\right) 
	= -2^{\alpha} x^{-\alpha} J_{\alpha+1}\left(x\right)
\end{equation}
where we used
\begin{equation*}
	\frac{d}{dx} \left(x^{-\alpha} J_{\alpha}\left(x\right)\right) = - x^{-\alpha} J_{\alpha+1}\left(x\right)
\end{equation*}
see for example \cite[Section 2.12, formula (6)]{Watson}. If we divide both sides in \eqref{eq:firstderivative} by $x$ and expand the sine function as $\sin\left(\frac{x}{n}\right) = \frac{x}{n} + O(x^3)$, we get
\begin{equation*}
	\lim_{n\to\infty} \frac{1}{n^{\alpha+2}} \frac{d}{dx}\left(P_s^{(\alpha,\beta)}\right)\left(\cos\left(\frac{x}{n}\right)\right) 
		= 2^{\alpha} x^{-\alpha-1} J_{\alpha+1}\left(x\right).
\end{equation*}
By repeating this argument, we obtain
\begin{equation}\label{eq:jthderivative}
	\lim_{n\to\infty} \frac{1}{n^{\alpha+2j}} \frac{d^{j}}{dx^{j}} \left(P_s^{(\alpha,\beta)} \right)\left(\cos\left(\frac{x}{n}\right)\right) 
		= 2^{\alpha} x^{-\alpha-j} J_{\alpha+j}\left(x\right),
		\qquad
		0\leq j \leq r.
\end{equation}
Hence, the limiting behavior of $\frac{1}{n^{\alpha+2r}} P_{\lambda, \mu,n}^{(\alpha,\beta)} \left( \cos \left(\frac{x}{n}\right) \right)$ is determined by the last term in \eqref{eq:MehlerHeine:1} only; for $j<r$, there is no contribution because of \eqref{eq:jthderivative}. This gives us
\begin{equation*}
	\lim_{n \to \infty} 
		\frac{1}{n^{\alpha +2r}} P_{\lambda, \mu,n}^{(\alpha,\beta)} \left( \cos\left(\frac{x}{n}\right) \right)
		= \Omega_{\lambda, \mu}^{(\alpha,\beta)}(1) \, 2^{\alpha+r_2} x^{-\alpha-r} J_{\alpha+r}(x)
\end{equation*}
as $\lim\limits_{n \to \infty}  \left(1+\cos\left(\frac{x}{n}\right)\right)^{r_2} \Omega_{\lambda, \mu}^{(\alpha,\beta)}\left(\cos \left(\frac{x}{n}\right)\right)= 2^{r_2} \cdot \Omega_{\lambda, \mu}^{(\alpha,\beta)}(1)$. This ends the proof of the first limit.

Next, we prove \eqref{eq:XJPMehlerHeineBeta} which treats the limiting behavior near the endpoint -1. We cannot directly use the approach we used to prove the previous limit. If we expand $P^{(\alpha,\beta)}_{\lambda,\mu,n}$ as in \eqref{eq:MehlerHeine:1}, we run into trouble because of the factor $(1+x)^{r_2}$. Moreover, the entries in the Wronskians consists of the factors $(1+x)^{-\beta}$. To avoid this problem, we first rewrite the exceptional Jacobi polynomial in terms of $\bar{P}^{(\alpha',\beta')}_{\lambda,\mu',n}(x)$, defined in \eqref{eq:Pbar}. Therefore, we use \eqref{lem:type2and3} where $\alpha'=\alpha+\mu_1+r_2$ and $\beta'=\beta-\mu_1-r_2$. Set $r'=r_1+r(\mu')$ where $r(\mu')$ is the length of the conjugated partition $\mu'$. 

Now, we can approach as before and use \eqref{eq:JacMehHeiBeta}, we have
\begin{equation*}
	\bar{P}^{(\alpha',\beta')}_{\lambda,\mu',n}(x) 
		= \sum_{j=0}^{r'-1}Q_{j}(x)\frac{d^j}{dx^j}P^{(\alpha',\beta')}_{s}(x) + (1-x)^{r(\mu')}
		\bar{\Omega}^{(\alpha',\beta')}_{\lambda,\mu'}(x) \frac{d^{r'}}{dx^{r'}}P^{(\alpha',\beta')}_{s}(x).
\end{equation*}
Via \eqref{eq:JacMehHeiBeta}, we obtain that the limiting behavior is determined by the last term, because \eqref{eq:JacMehHeiBeta} gives
\begin{equation*}
	\lim_{n\to\infty} \frac{(-1)^{s+j}}{n^{\beta+2j}} \frac{d^{j}}{dx^{j}} \left(P_s^{(\alpha',\beta')} \right)\left(-\cos\left(\frac{x}{n}\right)\right) 
		= 2^{\beta'} x^{-\beta'-j} J_{\beta'+j}\left(x\right),
		\qquad
		0\leq j \leq r'.
\end{equation*}
So, we get
\begin{equation*}
	\lim_{n \to \infty} 
		\frac{(-1)^{s+r'}}{n^{\beta' +2r'}} \bar{P}^{(\alpha',\beta')}_{\lambda,\mu',n} 
		= \bar{\Omega}_{\lambda, \mu'}^{(\alpha',\beta')}(-1) \, 2^{\beta'+r(\mu')} x^{-\beta'-r'} J_{\beta'+r'}(x).
\end{equation*}

Finally, we have to compare $\bar{P}^{(\alpha',\beta')}_{\lambda,\mu',n}$ with $P^{(\alpha,\beta)}_{\lambda,\mu,n}$. Lemma \ref{lem:type2and3} gives
\begin{equation*}
	\lim_{n \to \infty}  P^{(\alpha,\beta)}_{\lambda,\mu,n}\left(-\cos\left(\frac{x}{n}\right) \right)
		= \lim_{n \to \infty} c \cdot \bar{P}^{(\alpha+\mu_1+r_2,\beta-\mu_1-r_2)}_{\lambda,\mu',n}\left(-\cos\left(\frac{x}{n}\right) \right). 
\end{equation*}
Consider the constant term $c$, defined in \eqref{eq:constantc}. We split it into two parts: we have the factors which are independent of $n$ and the factors which depends on $n$ (recall that $s$ is in terms of $n$). We have
\begin{equation*}
	c
		= c_{ind} \cdot \frac{\prod\limits_{j=1}^{r_2}(s+\beta-m_j)}{\prod\limits_{j=1}^{r(\mu')}(s+\alpha'-m_j')} = O(n^{r_2-r(\mu')})
\end{equation*}
as $n$ tends to infinity because $c_{ind}\neq0$. We conclude that
\begin{equation*}
	\lim_{n \to \infty} 
		\frac{(-1)^{s+r'}}{n^{\beta' +2r'+r_2-r'(\mu)}} P^{(\alpha,\beta)}_{\lambda,\mu,n} 
		= c_{ind} \cdot \bar{\Omega}_{\lambda, \mu'}^{(\alpha',\beta')}(-1) \, 2^{\beta'+r(\mu')} x^{-\beta'-r'} J_{\beta'+r'}(x).
\end{equation*}
A calculation gives that $\beta'+r'=\beta+r_1-r_2$ and hence the result is obtained as $r(\mu')=\mu_1$.
\end{proof}

\subsection{Proof of Theorem \ref{thm:XJPArcsineLaw}: the weak macroscopic limit of the regular zeros}\label{sec:ProofArcsineLaw}
The limiting behavior of the zero-counting measure of the zeros of the Jacobi polynomial is given by the arcsine distribution, see \cite[Section 4.2]{Kuijlaars_Van_Assche}. This distribution is independent of the parameters $\alpha$ and $\beta$. 

\begin{theorem}\label{thm:JacArcsineLaw}
For any $\alpha,\beta>-1$, let $ 1 > a^{(\alpha,\beta)}_{1,n} > a^{(\alpha,\beta)}_{2,n} > \dots > b^{(\alpha,\beta)}_{n,n} > -1 $ denote the simple zeros of the Jacobi polynomial $P^{(\alpha,\beta)}_{n}$. Then, for any continuous function $f$ on $(-1,1)$, it holds that
\begin{equation}\label{eq:JacArsineLaw}
	\lim_{n\to\infty}\frac{1}{n}\sum_{j=1}^{n}f\left(a^{(\alpha,\beta)}_{j,n}\right)
	=\frac{1}{\pi}\int_{-1}^{1} \frac{f(x)}{\sqrt{1-x^2}}dx.
\end{equation}
\end{theorem}

This result extends directly to exceptional Jacobi polynomials as stated in Theorem \ref{thm:XJPArcsineLaw}. The argument for the proof of this result is completely analogous as in the Hermite \cite[Theorem 2.2]{Kuijlaars_Milson} and Laguerre \cite[Theorem 4]{Bonneux_Kuijlaars} case. For this reason, we only give the main idea.

\begin{proof}[Proof of Theorem \ref{thm:XJPArcsineLaw}]
From Theorem \ref{thm:XJPN(n)} and Corollary \ref{cor:BeardonDriver}, we know that if the degree tends to infinity, then the regular zeros follow the zeros of the Jacobi polynomial, except for a finite number of them. Hence \eqref{eq:JacArsineLaw} generalizes directly to \eqref{eq:XJPArcsineLaw}.
\end{proof}

\subsection{Proof of Theorem \ref{thm:XJPExceptionalZeros}: asymptotic behavior of  the exceptional zeros}\label{sec:ProofExceptionalZeros}
The weight function for exceptional Jacobi polynomials, stated in Lemma \ref{lem:XJPorth}, is given by
\begin{equation*}
	W^{(\alpha,\beta)}_{\lambda, \mu}(x) 
		= \frac{(1-x)^{\alpha+r_1+r_2}(1+x)^{\beta+r_1-r_2}}
		{\left(\Omega_{\lambda,\mu}^{(\alpha,\beta)}(x) \right)^2},
	\qquad
	x \in (-1,1).	 
\end{equation*}
In this section we view the above weight function as a meromorphic function in $\mathbb C \setminus \{-1,1\}$ with poles at the zeros of $\Omega_{\lambda,\mu}^{(\alpha,\beta)}$. To this end, we do not want that the generalized Jacobi polynomial vanish on $(-1,1)$ and therefore we assume conditions \eqref{Condition2GJP} and \eqref{Condition1GJP} to be satisfied. We start with the following property, which can be compared to Lemma 11 in \cite{Bonneux_Kuijlaars}. The proof follows the same ideas.

\begin{lemma}\label{lem:Residue}
For any partitions $\lambda$ and $\mu$, take $\alpha,\beta\in\mathbb{R}$ such that the conditions \eqref{Condition2GJP} and \eqref{Condition1GJP} are satisfied. Then, for every $n \in \mathbb N_{\lambda, \mu}$, the function
\begin{equation*}
	\left(P^{(\alpha,\beta)}_{\lambda,\mu,n} \right)^2 W^{(\alpha,\beta)}_{\lambda,\mu}
\end{equation*}
has zero residue at each of its poles in $\mathbb C \setminus \{-1,1\}$.
\end{lemma}
\begin{proof}
Without loss of generality, we may assume the conditions \eqref{Condition2XJP} and \eqref{Condition1XJP} to be satisfied instead of \eqref{Condition2GJP} and \eqref{Condition1GJP}. Otherwise, the exceptional Jacobi polynomial would vanish and hence there is nothing to prove as there would be no poles. 	

We apply a Darboux-Crum transformation to the differential operator \eqref{eq:DVJac2} with seed functions $\varphi_{n_1}^{(\alpha,\beta)}, \ldots, \varphi_{n_{r_1}}^{(\alpha,\beta)}, \varphi_{m_1}^{(\alpha,-\beta)}, \ldots, \varphi_{m_{r_2}}^{(\alpha,-\beta)}$, see Table \ref{tab:2}. This transformation leads to a new differential operator
\begin{equation} \label{eq:Oplamu} 
	y  \mapsto - y'' + V_{\lambda,\mu} y 
\end{equation}
with potential
\begin{align} 
	V_{\lambda,\mu}(x)
		= V(x)  - 
		2 \frac{d^2}{dx^2} \log \left(\Wr[\varphi^{(\alpha,\beta)}_{n_1},\dots, 
		\varphi^{(\alpha,\beta)}_{n_{r_1}}, \varphi^{(\alpha,-\beta)}_{m_1},\dots, \varphi^{(\alpha,-\beta)}_{m_{r_2}}]\right)
	\label{eq:Vlamu}
\end{align}
where $V(x)$ is defined in \eqref{eq:Potential}. The differential operator \eqref{eq:Oplamu} has eigenfunctions of the form
\begin{equation} \label{eq:OplamuEF}
	\frac{\Wr[\varphi^{(\alpha,\beta)}_{n_1},\dots, 
		\varphi^{(\alpha,\beta)}_{n_{r_1}}, \varphi^{(\alpha,-\beta)}_{m_1},\dots, \varphi^{(\alpha,-\beta)}_{m_{r_2}},\varphi^{(\alpha,\beta)}_s]}
	{\Wr[\varphi^{(\alpha,\beta)}_{n_1},\dots, 
		\varphi^{(\alpha,\beta)}_{n_{r_1}}, \varphi^{(\alpha,-\beta)}_{m_1},\dots, \varphi^{(\alpha,-\beta)}_{m_{r_2}}]}
\end{equation} 
where $s\geq 0$ and $s\neq n_j$ for every $j=1, \ldots, r_1$. Set $r=r_1+r_2$.
	
Using \eqref{eq:Wr2} and \eqref{eq:Wr1}, we can express	the Wronskian in \eqref{eq:Vlamu} as a Wronskian of the functions $f_1, \ldots, f_r$ defined in \eqref{eq:fj1}-\eqref{eq:fj2}. After a rather straightforward calculation, using elementary Wronskian properties, we end up with a prefactor and the polynomial \eqref{eq:OmegaLaMu},
\begin{multline*}
	\Wr[\varphi^{(\alpha,\beta)}_{n_1},\dots,\varphi^{(\alpha,\beta)}_{n_{r_1}}, \varphi^{(\alpha,-\beta)}_{m_1},\dots, \varphi^{(\alpha,-\beta)}_{m_{r_2}}]
		\\
		= \frac{(-4)^{\frac{r(r-1)}{2}}}{2^{\beta r_2}} \sin(x)^{\left(\alpha+\frac{1}{2}\right)r} \cos(x)^{\left(\beta+\frac{1}{2}\right)r} \cdot \Omega^{(\alpha,\beta)}_{\lambda,\mu}(\cos(2x)).
\end{multline*}
Hence the potential \eqref{eq:Vlamu} can be written as
\begin{multline*}
	V_{\lambda,\mu}(x) 
	= \frac{\left(\alpha-\frac{1}{2}+2r\right)\left(\alpha+\frac{1}{2}\right)}{\sin^2(x)} + \frac{\left(\beta-\frac{1}{2}+2r\right)\left(\beta+\frac{1}{2}\right)}{\cos^2(x)} +\frac{r(r-1)}{\sin^2(x)\cos^2(x)}  \\ -
	2 \frac{d^2}{dx^2}\log \left( \Omega^{(\alpha,\beta)}_{\lambda,\mu}(\cos(2x)) \right).
\end{multline*}
Similarly, the eigenfunctions \eqref{eq:OplamuEF} are
\begin{equation*} 
	(-4)^r \sin(x)^{\left(\alpha+\frac{1}{2}+r\right)} \cos(x)^{\left(\beta+\frac{1}{2}+r\right)}
	\cdot \frac{P^{(\alpha,\beta)}_{\lambda,\mu,n}(\cos(2x))}{\Omega^{(\alpha,\beta)}_{\lambda,\mu}(\cos(2x))}
\end{equation*}
where we choose $n \in \mathbb N_{\lambda,\mu}$ and $s=n-|\lambda|-|\mu|+r_1$.

Via the above described set-up, we derive that the function $ \left(P^{(\alpha,\beta)}_{\lambda,\mu,n} \right)^2 W^{(\alpha,\beta)}_{\lambda,\mu} $ has zero residue at each of its poles in $\mathbb C \setminus \{-1,1\}$. To establish this result, observe that the operator \eqref{eq:Oplamu} has trivial monodromy at every point $p \in \mathbb C \setminus \{-1,1\}$, see Proposition 5.21 in \cite{GarciaFerrero_GomezUllate_Milson}. Next, Proposition 3.3 in \cite{Duistermaat_Grunbaum} gives information about residues in this setting. As the precise arguments can be adapted from Lemma 11 in \cite{Bonneux_Kuijlaars}, the details are left out.
\end{proof}

Finally, we are able to prove the asymptotic behavior of the exceptional zeros. Again, the ideas of the proof are similar as the proof of Theorem 5 in \cite{Bonneux_Kuijlaars} or Theorem 2.3 in \cite{Kuijlaars_Milson}.

\begin{proof}[Proof of Theorem \ref{thm:XJPExceptionalZeros}]
Let $z_j$ be simple zero of $\Omega^{(\alpha,\beta)}_{\lambda,\mu}$ where 
$z_j\in\mathbb{C}\setminus[-1,1]$. Then $z_j$ is a double pole of $W_{\lambda,\mu}^{(\alpha,\beta)}$, and according to Lemma \ref{lem:Residue} we get for a certain constant $C_0\in\mathbb{C}$ that
\begin{equation} \label{eq:Taylorexpansion} 
	(x-z_j)^2 \left( L_{\lambda,\mu,n}^{(\alpha,\beta)}(x) \right)^2 
	W_{\lambda,\mu}^{(\alpha,\beta)}(x)
	= C_0 + O\left( (x-z_j)^2 \right), \qquad \text{as } x \to z_j.
\end{equation}
We may assume that $C_0 \neq 0$, otherwise $z_j$ is a zero of 
$P^{(\alpha,\beta)}_{\lambda,\mu,n}$ as well and then \eqref{eq:XJPExceptionalZeros} is clearly satisfied.

Since $C_0 \neq 0$, we can take an analytic logarithm of \eqref{eq:Taylorexpansion} 
in the neighborhood of $z_j$. The derivative of this logarithm vanishes because there is no residue, see Lemma \ref{lem:Residue}. Therefore, after a few calculations and simplifications, we arrive at the identity
\begin{equation}\label{eq:XJPExceptionalZeros:Proof:2}
	\sum_{k=1}^{N(n)}\frac{1}{z_j-x^{(\alpha,\beta)}_{k,n}}+\sum_{k=1}^{n-N(n)}\frac{1}{z_j-z^{(\alpha,\beta)}_{k,n}}
	=\frac{(\alpha+r)}{2(1-z_j)} - \frac{(\beta+r)}{2(1+z_j)} + \frac{3z_j}{1-z_j^2}
	+ \sum_{\shortstack{$\scriptstyle k=1 $ \\ $ \scriptstyle k\neq j$}}^{|\lambda|+|\mu|}\frac{1}{z_j-z_k}
\end{equation}
where $1 > x_{1,n}^{(\alpha,\beta)} \geq \cdots \geq x^{(\alpha,\beta)}_{N(n),n}>-1$ 
are the regular zeros and $z^{(\alpha,\beta)}_{1,n},\dots,z^{(\alpha,\beta)}_{n-N(n),n}$
are the exceptional zeros of the exceptional Jacobi polynomial and
$z_1, \ldots, z_{|\lambda|+|\mu|}$ are the zeros of $\Omega_{\lambda,\mu}^{(\alpha,\beta)}$.

Now we can proceed completely similar as in the proof of Theorem 5 in \cite{Bonneux_Kuijlaars} (or the proof of Theorem 2.3 in \cite{Kuijlaars_Milson}). We only give the ideas.
\begin{enumerate}
	\item 
	By orthogonality, all $n$ zeros of the Jacobi polynomial $P^{(\alpha+r,\beta+r)}_n$ are in the open interval $(-1,1)$. Because of Corollary \ref{cor:BeardonDriver}, the number of zeros of the exceptional Jacobi polynomial $P^{(\alpha,\beta)}_{\lambda,\mu,n}$ in $(-1,1)$ grows like $n$. Hence, the number of elements in the first sum of the left hand side in \eqref{eq:XJPExceptionalZeros:Proof:2} grows like $n$ for $n$ large enough.
	
	\item 
	The right hand side of \eqref{eq:XJPExceptionalZeros:Proof:2} is independent of $n$.
	
	\item 
	Distinguish two cases: 
	\begin{itemize}
		\item[(a)]
		If $z_j\notin\mathbb{R}$, then consider the imaginary part of both sides in the identity \eqref{eq:XJPExceptionalZeros:Proof:2}.
		
		\item[(b)]
		If $z_j\in\mathbb{R}$, then consider the real part of both sides in the identity \eqref{eq:XJPExceptionalZeros:Proof:2}. 
	\end{itemize}
	By the first item, a short calculation gives that the real/imaginary part of the first sum in the left hand side grows like $-n$. By the second item, the second sum in the left hand has to compensate the first sum to establish the identity. As this second sum is finite, at least 1 element has to grow like $n$. Hence there exists a $k$, such that for $n$ large enough, 
	\begin{equation*}
		\left|\im\left(\frac{1}{z_j-z^{(\alpha,\beta)}_{k,n}}\right)\right| > c \cdot n \quad \text{ or } \quad 
		\left|\re\left(\frac{1}{z_j-z^{(\alpha,\beta)}_{k,n}}\right)\right| > c \cdot n
	\end{equation*}
	for some $c>0$. This results that $z_j$ attracts $z_{k,n}^{(\alpha,\beta)}$ for some $k$ at speed $O\left(n^{-1}\right)$. 
\end{enumerate}
\end{proof}

\section*{Acknowledgements}
The author thanks Arno Kuijlaars for fruitful discussions and a careful reading of a preliminary version of this article. The author also thanks the referees for their useful remarks.

The author is supported in part by the long term structural funding-Methusalem grant of the Flemish Government, and by EOS project 30889451 of the Flemish Science Foundation (FWO).


\begin{thebibliography}{99}
\bibitem{Andrews} Andrews  G.E.,
The theory of partitions,
Cambridge University Press, New York, 1998.

\bibitem{Askey} Askey R.,
Orthogonal Polynomial and Special Functions,
Society for Industrial and Applied Mathematics, Philadelphia, Pennsylvania, 1975.

\bibitem{Bagchi_Grandati_Quesne} Bagchi B., Grandati Y. and Quesne C.,
Rational extensions of the trigonometric Darboux-Poschl-Teller potential based on para-Jacobi polynomials,
\emph{Journal of Mathematical Physics} 56 (2015), 062103.

\bibitem{Beardon_Driver} Beardon A.F. and Driver K.A.,
The zeros of linear combinations of orthogonal polynomials,
\emph{Journal of Approximation Theory} 137 (2005), 179--186.

\bibitem{Grandati_Berard} B{\'e}rard A. and Grandati Y.,
Comments on the generalized {SUSY QM} partnership for {D}arboux--{P}{\"o}schl--{T}eller potential and exceptional {J}acobi polynomials,
\emph{Journal of Engineering Mathematics} 82 (2013), 161--171.

\bibitem{Bochner} Bochner S.,
{\"U}ber Sturm-Liouvillesche polynomsysteme,
\emph{Mathematische Zeitschrift} 29 (1929), 730--736.

\bibitem{Bonneux_Kuijlaars} Bonneux N. and Kuijlaars A.B.J.,
Exceptional Laguerre polynomials,
\emph{Studies in Applied Mathematics} 141 (2018), 547--595.
	
\bibitem{Crum} Crum M.M.,
Associated {S}turm-{L}iouville systems,
\emph{The Quarterly Journal of Mathematics} 6 (1955), 121--127.
	
\bibitem{Curbera_Duran} Curbera G.P. and Dur{\'a}n A.J.,
Invariant properties for {W}ronskian type determinants of classical and classical discrete orthogonal polynomials under an involution of sets of positive integers,
preprint arXiv:1612.07530.
	
\bibitem{Darboux} Darboux G.,
Sur une proposition relative aux \'equations lin\'eaires,
\emph{Comptes Rendus de l'Acad\'emie des Sciences} 94 (1882), 1456--1459.

\bibitem{Dimitrov_Lun} Dimitrov D.K. and Lun Y. Ch.,
Monotonicity, interlacing and electrostatic interpretation of zeros of exceptional Jacobi polynomials,
\emph{Journal of Approximation Theory} 181 (2014), 18--29.
	
\bibitem{Duistermaat_Grunbaum} Duistermaat J.J. and Gr\"unbaum F.A.,
Differential equations in the spectral parameter,
\emph{Communications in Mathematical Physics} 103 (1986), 177--240.

\bibitem{Duran-a} Dur{\'a}n A.J.,
Constructing bispectral dual {H}ahn polynomials,
\emph{Journal of Approximation Theory} 189 (2015), 1--28.

\bibitem{Duran-Hermite} Dur{\'a}n A.J.,
Exceptional Charlier and Hermite orthogonal polynomials,
\emph{Journal of Approximation Theory} 182 (2014), 29--58.

\bibitem{Duran-b} Dur{\'a}n A.J.,
Exceptional {H}ahn and {J}acobi orthogonal polynomials,
\emph{Journal of Approximation Theory} 214 (2017), 9--48.

\bibitem{Duran-Laguerre1} Dur{\'a}n A.J.,
Exceptional {M}eixner and {L}aguerre orthogonal polynomials,
\emph{Journal of Approximation Theory} 184 (2014), 176--208.

\bibitem{Duran_Perez-Laguerre2} Dur{\'a}n A.J. and P{\'e}rez M.,
Admissibility condition for exceptional {L}aguerre polynomials,
\emph{Journal of Mathematical Analysis and Applications} 424 (2015), 1042--1053.

\bibitem{Erdelyi} Erd\'elyi A.,
Higher Transcendental Functions, Volume 1,
McGraw-Hill book company, New York, 1953.

\bibitem{Felder_Hemery_Veselov} Felder G., Hemery A.D. and Veselov A.P.,
Zeros of {W}ronskians of {H}ermite polynomials and {Y}oung diagrams,
\emph{Physica D: Nonlinear Phenomena} 241 (2012), 2131--2137.

\bibitem{GarciaFerrero_GomezUllate} Garc{\'i}a-Ferrero M. and G{\'o}mez-Ullate D.,
Oscillation theorems for the {W}ronskian of an arbitrary sequence of eigenfunctions of {S}chr{\"o}dinger’s equation,
\emph{Letters in Mathematical Physics} 105 (2015), 551-–573.

\bibitem{GarciaFerrero_GomezUllate_Milson} Garc{\'i}a-Ferrero M., G{\'o}mez-Ullate D. and Milson R.,
A {B}ochner type classification theorem for exceptional orthogonal polynomials, 
to appear in \emph{Journal of Mathematical Analysis and Applications}.

\bibitem{GomezUllate_Grandati_Milson-a} G{\'o}mez-Ullate D., Grandati Y. and Milson R.,
Durfee rectangles and pseudo-{W}ronskian equivalences for {H}ermite polynomials,
\emph{Studies in Applied Mathematics} 141 (2018), 596--625.

\bibitem{GomezUllate_Grandati_Milson-b} G{\'o}mez-Ullate D., Grandati Y. and Milson R.,
Rational extensions of the quantum harmonic oscillator and exceptional Hermite polynomials,
\emph{Journal of Physics A: Mathematical and Theoretical} 47 (2014), 015203.

\bibitem{GomezUllate_Grandati_Milson-L+J} G{\'o}mez-Ullate D., Grandati Y. and Milson R.,
Shape invariance and equivalence relations for pseudowronskians of Laguerre and Jacobi polynomials,
\emph{Journal of Physics A: Mathematical and Theoretical} 51 (2018), 345201.

\bibitem{GomezUllate_Kamran_Milson-09} G{\'o}mez-Ullate D., Kamran N. and Milson R.,
An extended class of orthogonal polynomials defined by a {S}turm-{L}iouville problem,
\emph{Journal of Mathematical Analysis and Applications} 359 (2009), 352--367.

\bibitem{GomezUllate_Kamran_Milson-XJacobi} G{\'o}mez-Ullate D., Kamran N. and Milson R.,
On orthogonal polynomials spanning a non-standard flag,
\emph{Contemporary Mathematics} 563 (2012), 51--71.

\bibitem{GomezUllate_Marcellan_Milson} G{\'o}mez-Ullate D., Marcell{\'a}n F. and Milson R.,
Asymptotic and interlacing properties of zeros of exceptional {J}acobi and {L}aguerre polynomials,
\emph{Journal of Mathematical Analysis and Applications} 399 (2013), 480--495.

\bibitem{Grandati_Quesne} Grandati Y. and Quesne C.,
Confluent chains of DBT: Enlarged shape invariance and new orthogonal polynomials,
\emph{Symmetry, Integrability and Geometry: Methods and Applications} 11 (2015), 061.

\bibitem{Ho_Odake_Sasaki} Ho C.L., Odake S. and Sasaki R.,
Properties of the exceptional {$(X_l)$} {L}aguerre and {J}acobi polynomials,
\emph{Symmetry, Integrability and Geometry: Methods and Applications} 7 (2011), 107.

\bibitem{Ho_Sasaki} Ho C.L. and Sasaki R.,
Zeros of the exceptional {L}aguerre and {J}acobi Polynomials,
\emph{International Scholarly Research Notices: Mathematical Physics} 2012 (2012), 27 pp.

\bibitem{Ho_Sasaki_Takemura} Ho C.L., Sasaki R. and Takemura K.,
Confluence of apparent singularities in multi-indexed orthogonal polynomials: the {J}acobi case,
\emph{Journal of Physics A: Mathematical and Theoretical} 46 (2013), 115205.

\bibitem{Horvath} Horv{\'a}th {\'A}.P.,
The electrostatic properties of zeros of exceptional {L}aguerre and {J}acobi polynomials and stable interpolation,
\emph{Journal of Approximation Theory} 194 (2015), 87--107.

\bibitem{Kuijlaars_Martinez_Orive} Kuijlaars  A.B.J., Mart{\'i}nez-Finkelshtein A. and Orive R.,
Orthogonality of Jacobi polynomials with general parameters,
\emph{Electronic Transactions on Numerical Analysis} 19 (2005), 1--17.

\bibitem{Kuijlaars_Milson} Kuijlaars A.B.J. and Milson R.,
Zeros of exceptional {H}ermite polynomials,
\emph{Journal of Approximation Theory} 200 (2015), 28--39.

\bibitem{Kuijlaars_Van_Assche} Kuijlaars A.B.J. and Van Assche W.,
The asymptotic zero distribution of orthogonal polynomials with varying recurrence coefficients,
\emph{Journal of Approximation Theory} 99 (1999), 167--197.

\bibitem{Lesky} Lesky P.,
Die charakterisierung der klassischen orthogonalen polynome durch Sturm-Liouvillesche differentialgleichungen,
\emph{Archive for Rational Mechanics and Analysis} 10 (1962), 341--351.

\bibitem{Liaw_Littlejohn_Stewart} Liaw C., Littlejohn L.L. and Stewart J., 
Spectral analysis for the exceptional $X_m$-{J}acobi equation,
\emph{Electronic Journal of Differential Equations} 194 (2015), 10 pp.

\bibitem{Liaw_Littlejohn_Stewart_Wicks} Liaw C., Littlejohn L.L., Stewart J. and Wicks Q., 
A spectral study of the second-order exceptional {$X_1$}-{J}acobi differential expression and a related non-classical {J}acobi differential expression,
\emph{Journal of Mathematical Analysis and Applications} 422 (2015), 212--239.

\bibitem{Liaw_Stewart_Osborn} Liaw C., Stewart J. and Osborn J., 
Moment representations of exceptional $X_1$ orthogonal polynomials,
\emph{Journal of Mathematical Analysis and Applications} 455 (2017), 1848--1869.

\bibitem{Midya_Roy} Midya B. and Roy B.,
Infinite families of (non)-{H}ermitian {H}amiltonians associated with exceptional {$X_m$} {J}acobi polynomials,
\emph{Journal of Physics A: Mathematical and Theoretical} 46 (2013), 175201.

\bibitem{Odake_Sasaki-a} Odake S. and Sasaki R.,
Multi-indexed {M}eixner and little q-{J}acobi ({L}aguerre) polynomials,
\emph{Journal of Physics A: Mathematical and Theoretical} 50 (2017), 165204.

\bibitem{Odake_Sasaki-b} Odake S. and Sasaki R.,
Simplified expressions of the multi-indexed {L}aguerre and {J}acobi polynomials,
\emph{Symmetry, Integrability and Geometry: Methods and Applications} 13 (2017), 020.

\bibitem{Sasaki_Tsujimoto_Zhedanov} Sasaki R., Tsujimoto S. and Zhedanov A.,
Exceptional {L}aguerre and {J}acobi polynomials and the corresponding potentials through {D}arboux-{C}rum transformations,
\emph{Journal of Physics A: Mathematical and Theoretical} 43 (2010), 315204.

\bibitem{Szego} Szeg\H{o} G.,
Orthogonal Polynomials, 4th edition,
American Mathematical Society, Providence, Rhode Island, 1975.	

\bibitem{Takemura-Heun} Takemura K.,
Heun's equation, generalized hypergeometric function and exceptional {J}acobi polynomial,
\emph{Journal of Physics A: Mathematical and Theoretical} 45 (2012), 085211.

\bibitem{Takemura-Maya} Takemura K.,
Multi-indexed {J}acobi polynomials and {M}aya diagrams,
\emph{Journal of Mathematical Physics} 55 (2014), 113501.	

\bibitem{Shen_Tang_Wang} Shen J., Tang T. and Wang L.L.,
Orthogonal polynomials and related approximation results,
in: Spectral methods: algorithms, analysis and applications,
Springer Verlag, Berlin, 2011, 47--140.

\bibitem{Watson} Watson G.,
A {T}reatise on the {T}heory of {B}essel {F}unctions,
Cambridge University Press, New York, 1944.

\end{thebibliography}
\end{document}